\documentclass[11pt,reqno]{amsart}

\usepackage{amsmath,amssymb,mathrsfs}
\usepackage{graphicx,cite,times}
\usepackage{hyperref}
\hypersetup{
    colorlinks=true,
    linkcolor=blue,
    filecolor=gray,
    urlcolor=blue,
    citecolor=blue,
}

\usepackage[doipre={doi:~}]{uri}
\usepackage{cases}
 \usepackage{dsfont}
 \usepackage[mathscr]{eucal}
 \usepackage{color}
\newtheorem{theorem}{\color{black}\indent Theorem}[section]

\newtheorem{lemma}{\color{black}\indent Lemma}[section]
\newtheorem{remark}{\color{black}\indent Remark}

\begin{document}

\title{KAM Theorems for Multi-scale Torus}

\author{Weichao Qian}
\address{School of Mathematics, Jilin University, Changchun, Jilin 130012, P.R.China }
\email{qianwc@163.com}
\author{Yixian Gao}
\address{School of
Mathematics and Statistics, Center for Mathematics and
Interdisciplinary Sciences, Northeast Normal University, Changchun, Jilin 130024, P.R.China. }
\email{gaoyx643@nenu.edu.cn}

\author{Yong  Li}
\address{School of Mathematics, Jilin University, Changchun, Jilin 130012;
School of Mathematics and Statistics, Center for Mathematics and
Interdisciplinary Sciences, Northeast Normal University, Changchun, Jilin
130024, P.R.China}
\email{yongli@nenu.edu.cn}

\thanks{ The research of WQ was partially supported by China Postdoctoral Science Foundation (2021M701396, 2022T150262), NSFC grant(12201243), ERC (PR1062ERC01).
The research of YG was partially supported by NSFC grants (
11871140 and 12071065) and FRFCU2412019BJ005.
 The research of YL was partially supported by  NSFC grants (12071175, 11171132), Project of Science and Technology Development of Jilin Province (2017C028-1, 20190201302JC), and Natural Science Foundation of Jilin Province
(20200201253JC)}



\date{}

\begin{abstract}
In present paper, from the viewpoint of physical intuition we introduce a Hamiltonian system with multiscale rotation, which describes many systems, for example, the forced pendulum with fast rotation, weakly coupled $N$-oscillators with quasiperiodic force and so on. We study the persistence of invariant tori for this Hamiltonian system, and establish some KAM type results including the isoenergetic type. As consequences,
we can show that Boltzmann's ergodicity hypothesis is also not true for this Hamiltonian system.
\end{abstract}

\maketitle

\section{Introduction}
 Possessing high, low or mixed frequencies  systems are ubiquitous in different fields of science and engineering. For example,
the  Hindmarsh-Rose neuron model, mixed forcing currents, which are composed of low-frequency,
high-frequency and constant signals, are imposed on
the neuron \cite{hindmarsh1982model}. Especially, some problems arising slow rates of thermalization in statistical mechanics, which catch lots of attention, can be understood by studying the Hamiltonian system with multi-scale frequencies (\cite{Bambusi,Benettin1,Benettin2,Carati,Gauckler,Hairer,Shudo}).

In the high-frequency case, \cite{Gauckler} considered a Hamiltonian system of the following form:
\begin{eqnarray}\label{592}
H(p,q) = \sum_{j=1}^n \frac{1}{2} (|p_j|^2 + \omega_j^2 |q_j|^2 ) + \frac{1}{2} |p_0|^2 + U(q),
\end{eqnarray}
whose equations of motion are
\begin{eqnarray}\label{5101}
\ddot{q}_j + \omega_j^2 q_j  = - \nabla _j U(q), \quad  j= 0, \cdots ,n,
\end{eqnarray}
where momenta $p = (p_0, p_1, \cdots, p_n)$, position $q = (q_0, q_1,\cdots,q_n)$, $p_j$, $q_j \in \mathbb R^{1}$, $\omega_j \geq \frac{1}{\varepsilon}$, $j\geq 1$, $0<\varepsilon \ll 1$, $\omega_0 = 0$, $\nabla_j$ denotes the partial derivative with respect to $q_j$ and $U(q)$ denotes a special coupling potential. If $U(q)$ is smooth with derivatives bounded independently of the parameter $\varepsilon$, the solution $q(t)$ of motion equation(\ref{5101}) is of rapid rotation, which means the period in angle is small enough.
Usually Hamiltonian systems with action-angle variables are $2\pi-$periodic in angle. However, from the viewpoint of physical intuition, one should study Hamiltonian systems with rapid, slow, or multiscale rotation variables, i.e. a nearly integrable real analytic Hamiltonian of the following form:
\begin{eqnarray}\label{5102}
H(\lambda x, y) = N(y) + \varepsilon P(\lambda x, y),
\end{eqnarray}
where $x \in \mathbb T^n$, $y \in G\subset \mathbb R^n$, $G$ is a bounded region in $\mathbb R^n$, $\varepsilon$ is a small parameter, $\lambda = \varepsilon^\alpha$, $\alpha \in\mathbb R^1$. Certainly, system (\ref{5102}) is, respectively, fast, usual or slow one corresponding to $\alpha<0,=0,$ or $>0$. And the present paper will show the persistence of invariant tori for Hamiltonian system (\ref{5102}), i.e. a Hamiltonian system with fast, usual or slow rotation. Furthermore, we will study the persistence of invariant tori for Hamiltonian systems with multiscale rotation perturbations.

A simple model is the forced pendulum with fast rotation (fast motion, or rapidly forced periodic term in time):
$$
\ddot q+\sin q=\varepsilon \sin \frac {t}{\varepsilon}
$$
with the Hamiltonian
$$
H=\langle \eta,\omega\rangle+\frac12\dot q^2-\cos q+\varepsilon q \cos\frac{\omega\theta}{\varepsilon}
$$
by setting $\omega\theta=t$ and $\eta$ as the action variable, where $\omega,$ $\theta$ $\in$ $\mathbb R^1$. For this kind of systems with general form
$$
H=H_0(x,y)+\varepsilon^\alpha H_1(x,y,\frac t{\varepsilon},\varepsilon),
$$
some remarkable developments have been made for dynamics, such as average principle, adiabaticity, in particular, exponentially small splitting of separatrices and so on(see\cite{Arnold1,Arnold2,Baldoma1,Baldoma2,Delshams,Delshams2,Delshams3,Gelfreich,Guardia,Neishtadt2,Neishtadt4,Neishtadt6,Neishtadt3,Neishtadt1}).

A classical method to prove the existence of invariant tori was given by Kolmogorov (\cite{Kolmogorov}), Arnold (\cite{Arnold}) and Moser (\cite{Moser}), called KAM theory, which under the Kolmogorov's nondegenerate condition, i.e. $\partial_y^2 N(y) \neq 0$, shows the persistence of invariant tori of a real analytic nearly integrable Hamiltonian system of the following form:
\begin{eqnarray}\label{5103}
H(x,y) = N(y) + \varepsilon P(x,y)
\end{eqnarray}
where $x \in\mathbb T^d$, $y \in G \subset\mathbb R^d$, $G$ is a bounded closed region in $\mathbb R^d$. The KAM theory has been studied from different viewpoints and with different mathematical techniques in numerous publications. For developments to lower dimensional invariant tori, we refer the reader to \cite{Cong,Cheng,Livia,Gallavotti1,Gentile,Jorba,Li,Li1,Treshchev,Zehnder,Zehnder1}. For applications of the KAM theory to celestial mechanics, we refer the reader to \cite{han,Meyer,Siegel,Xu2,Xu}.

In the viewpoint of average method (\cite{Neishtadt}) Hamiltonian system (\ref{5103}) is a system with fast angle variables, in other words, the change of the angle is fast, which is different from the rapid rotation variables mentioned above. Actually, in this paper what we are concerned with are a series of more general Hamiltonian systems, whose periodicity in angle maybe very small, very large or of $2\pi$. The $2\pi-$period in angle is the main task of the standard KAM theory. But when the periodicity in angle is not $2\pi$, especially very small or large, to our knowledge, the study of the persistence of invariant tori seems very rare.

As we all known, the KAM theory is a method to use the fast Newton's iteration consisted of infinite steps, in which parts of the small perturbation are eliminated by the Lie derivative, to get a smaller perturbation, in which the small denominator should be controlled. If the period of the angle is very small, i.e., the perturbation of the Hamiltonian is $P(\varepsilon^{-\alpha} x, y)$, $\alpha \in\mathbb R^1_+$, then the coefficient of the Lie derivation will be large enough, which is beneficial to control the small denominator. It seems that the rapid rotation is good for eliminating the small perturbation, and in Section \ref{5105} we show that is indeed true. However, this is not the case of  slow rotation variables, i.e., $P(\varepsilon^{\beta} x, y)$, $\beta \in (0,1)$ because the coefficient of the Lie derivation will be small, which is harmful to control the small divisor. It seems that the slow rotation is worse to clean up the small perturbation and in Section \ref{lower} we show a upper bound for $\beta$, under which the small perturbation can be purged. More complex case is about the multiscale rotation perturbation, i.e., $P(\varepsilon^{-\alpha}x, y, \theta, \eta, \varepsilon^\beta \varphi,I)$, where the Lie derivation controlled by $\varepsilon^{-\alpha}$ and $ \varepsilon^{\beta}$ may be large enough, a constant or small enough, and in Section \ref{couple} we show the persistence of invariant tori under such a multiscale perturbation.

Concretely, we consider a Hamiltonian system of the following form:
\begin{eqnarray}\label{5131}
H^d(x,y,\theta, \eta, \varphi, I) = N^d(y,\eta,I) + \varepsilon P^d(\frac{x}{\lambda_1}, y,\theta, \eta, \lambda_2 \varphi, I),
\end{eqnarray}
defined on the complex neighborhood
\[D^d(r,s) = \{(x, y, \theta, \eta, \varphi, I): |{\rm Im}~ x| < r, |y|<s, |{\rm Im}~ \theta| < r, |\eta|<s, |{\rm Im}~ \varphi| < r, |I|<s\}
\]of $\mathbb T^d \times \{0 \} \times\mathbb T^d \times \{0 \}\times\mathbb T^d \times \{0 \} \subset\mathbb T^d \times\mathbb R^d \times\mathbb T^d \times\mathbb R^d \times\mathbb T^d \times\mathbb R^d$, where $\lambda_1 = \varepsilon^{\alpha}$, $\lambda_2 = \varepsilon^\beta,$ $\alpha \in \mathbb R_+^1$, $\beta \in (0, \frac{\sigma^2}{3[(d+m+5)\tau + d+ 2m +13]})$, $\sigma,$ $\tau$ and $m$ are defined in Theorem \ref{dingli15}, $N^d (y, \eta, I)$ is a real analytic function on a complex neighborhood of the bounded closed region $G^d \subset \mathbb R^d \times\mathbb R^d \times\mathbb R^d$; $\varepsilon P^d(\frac{x}{\lambda_1}, y, \theta, \eta, \lambda_2 \varphi, I)$, a small perturbation, is a real analytic function, where $\varepsilon > 0$ is a small parameter.

\begin{remark}
The perturbation in (\ref{5131}) maybe some terms mixed by all the variables of $\frac{x}{\lambda_1},$ $y,$ $\theta,$ $\eta,$ $\lambda_2 \varphi,$ $I$ or some terms with the form of $ P^a (\frac{x}{\lambda_1}, y)+  P^b (\lambda_2 \varphi, I)+  P^c(\theta, \eta)$.
\end{remark}

With the transformation: $y\rightarrow y +\xi_1, ~x \rightarrow x,~\eta \rightarrow \eta +\xi_2, ~\theta \rightarrow \theta,~I \rightarrow I +\xi_3, ~\varphi \rightarrow \varphi$,  where $ (\xi_1, \xi_2, \xi_3)\in G^d$, Hamiltonian system (\ref{5131}) reads
\begin{eqnarray}\label{5132}
H^d &=& N^d + \varepsilon P^d(\frac{x}{\lambda_1}, y,\theta, \eta, \lambda_2 \varphi, I, \xi_1, \xi_2, \xi_3), \\
\nonumber N^d &=& e^d + \langle \mathbf{a}, \mathbf{b}\rangle + \frac{1}{2} \langle \mathbf{b}, \mathfrak{A}^d \mathbf{b}\rangle + \hat{h}^d(y,\eta,I),
\end{eqnarray}
where $\hat{h}^d(y,\eta,I)$ are all terms with the form of $y^{\iota_1} \eta^{\iota_2} I^{\iota_3}$, $|\iota_1|+ |\iota_2|+|\iota_3|\geq 3$, in $N^d$, $\lambda_1 = \varepsilon^\alpha$, $\lambda_2 = \varepsilon^\beta$, $\alpha\in\mathbb R_+^1$, $\beta \in (0, \frac{\sigma^2}{3[(d+m+5)\tau + d+ 2m +13]})$, $\sigma,$ $\tau$ and $m$ are defined in Theorem \ref{dingli15},
\begin{eqnarray*}
\mathbf{a}= \left(
              \begin{array}{c}
                \omega^d \\
                \Lambda^d \\
                \Omega^d \\
              \end{array}
            \right)= \left(
                       \begin{array}{c}
                         \partial_y N^d \\
                         \partial_\eta N^d \\
                         \partial_I N^d\\
                       \end{array}
                     \right)
            ,~~\mathbf{b} = \left(
              \begin{array}{c}
                y \\
                \eta \\
                I \\
              \end{array}
            \right),~~\mathfrak{A}^d  =
                  \left(
                    \begin{array}{ccc}
                       \frac{\partial^ 2 N^d}{\partial y^2} & \frac{\partial^2 N^d}{\partial y \partial\eta} & \frac{\partial^ 2 N^d}{\partial y \partial I} \\
                     \frac{\partial^2N^d}{\partial \eta \partial y} &   \frac{\partial^2N^d}{\partial \eta^2} & \frac{\partial^2N^d}{\partial \eta \partial I}\\
                      \frac{\partial^2 N^d}{\partial I \partial y} & \frac{\partial^2 N^d}{\partial I \partial \eta} & \frac{\partial^2 N^d}{\partial I^2} \\
                    \end{array}
                  \right).
\end{eqnarray*}

To state our main results, we make the following assumptions.
\begin{itemize}
\item[\bf{(R)}]There exists an $N > 1$ such that
\begin{eqnarray*}
{\rm rank} \{ \partial_{\mathbf{b}}^\alpha \mathbf{a}: 0\leq |\alpha|\leq N,~~~~ \forall \mathbf{b} \in G^d \} = 3d.
\end{eqnarray*}
\item[\bf{(K)}] $\mathfrak{A}^d$ has an  $n \times n$ order nonsingular minor $\mathbf{\mathcal{A}}^d$.
\item[$\bf{(Iso)}$]
${\rm rank} \left(
  \begin{array}{cc}
    \mathfrak{A}^d & \mathbf{a} \\
    \mathbf{a}^{T} & 0 \\
  \end{array}
\right) = n+1$, $\forall \mathbf{b} \in G^d.$
\end{itemize}

When the perturbation of systems (\ref{5132}) is equal to zero, the existence of quasiperiodic solutions is obvious. What we are engaged to do by iteration is to show the persistence of quasiperiodic solutions for Hamiltonian system (\ref{5132}).

Our main results for (\ref{5131}) state as follows.
\begin{theorem}\label{dingli15}
Let $H^d$ be analytic and $\alpha\in \mathbb R_+^1$, $\beta \in (0, \frac{\sigma^2}{3[(d+m+5)\tau + d+ 2m +13]})$, where $\sigma \in (0, \frac{1}{3})$, $1\leq m <\infty$, $d(d - 1) -1<\tau < \infty$ are given.
\begin{itemize}
  \item [(1)] Assume $\bf{(R)}$ hold. Then there exist a $\varepsilon_0 >0 $ and a family of Cantor sets $G_\varepsilon^d \subset G^d$, $0<\varepsilon < \varepsilon_0$, such that for each $(y, \eta, I) \in G_\varepsilon^d$, the unperturbed $3d-$tours $T_{(y, \eta, I)}^d$ persists and gives rise to a real analytic, invariant $3d-$torus $T_{\varepsilon, (y, \eta, I)}^d$ of the perturbed system. Moreover, the relative Lebesgue measure $|G^d \setminus G_\varepsilon^d|$ tends to 0 as $\varepsilon \rightarrow 0$.
  \item [(2)] Assume $\bf{(R)}$  and $\bf{(K)}$ hold. Then there exist a $\varepsilon_0 >0 $ and a family of Cantor sets $G_\varepsilon^d \subset G^d$, $0<\varepsilon < \varepsilon_0$, such that for each $(y, \eta, I) \in G_\varepsilon^d$, the unperturbed $3d-$tours $T_{(y, \eta, I)}^d$ persists and gives rise to a real analytic, invariant $3d-$torus $T_{\varepsilon, (y, \eta, I)}^d$ preserving $n$ corresponding unperturbed toral frequencies. Moreover, the relative Lebesgue measure $|G^d \setminus G_\varepsilon^d|$ tends to 0 as $\varepsilon \rightarrow 0$.
  \item [(3)] Denote $\mathcal{M} = \{(y, \eta, I): H_1(y, \eta, I)=c\}$ by a given energy surface. Assume $\bf{(R)}$, $\bf{(K)}$ and $\bf{(Iso)}$ hold. Then there exist a $\varepsilon_0 >0 $ and a family of Cantor sets $\mathcal{M}_\varepsilon^d \subset \mathcal{M}$, $0<\varepsilon < \varepsilon_0$, such that for each $(y, \eta, I) \in \mathcal{M}_\varepsilon$, the unperturbed $3d-$tours $T_{(y, \eta, I)}^d$ persists and gives rise to a real analytic, invariant $3d-$torus $T_{\varepsilon, (y, \eta, I)}^d$ keeping the same energy and maintaining $n$ frequency ratios.  Moreover, the relative Lebesgue measure $|G^d \setminus G_\varepsilon^d |$ tends to 0 as $\varepsilon \rightarrow 0$.
\end{itemize}

\end{theorem}

\begin{remark}
Most of results in the KAM theory are about the Hamiltonian with the action-angle variables, and a bridge between momenta-position variables and action-angle variables is the symplectic coordinate transformation. Usually the standard symplectic coordinate transformation is $p_j = \sqrt{2I} \cos \varphi_j$, $q_j = \sqrt{\frac{2I}{\omega_j}} \sin \varphi_j$ in \cite{Benettin1,Benettin2}, which correspond the property of high-frequency to fast action variables.  And in the viewpoint of the average method the reduced Hamiltonian system is a system with fast angle variables, in other words, the change of the angle is fast. Obviously, the transformation loses the property of the rapid rotation. Then a symplectic transformation kept property of rapid rotation of the Hamiltonian system is necessary, and we refer the reader to $p = \sqrt{\frac{2y}{\sqrt{\omega_j}}} \cos \sqrt{\omega_j} x_j,$ $q  = \sqrt{\frac{2y}{\sqrt{\omega_j}}}\sin \sqrt{\omega_j} x_j$. A symplectic transformation keeping the property of slow rotation is similar.
\end{remark}

\begin{remark}
The first part of this theorem asserts the existence of invariant tori for nearly integrable Hamiltonian systems with multi-scale rotation perturbation.
\end{remark}

\begin{remark}
The second part of this theorem states the partial preservation of frequency for nearly integrable Hamiltonian system with multi-scale rotation perturbation. For similar results on classical nearly integrable Hamiltonian, refer to \cite{LLL,Sevryuk}.
\end{remark}

\begin{remark}
The third part of Theorem \ref{dingli15} implies that generally isoenergetic Boltzmann's ergodicity hypothesis does not hold for Hamiltonian system (\ref{5131})
\end{remark}

The present paper be arranged as following. In Section \ref{5104}, we show the iteration sequences, which is necessary for the KAM iteration, and some notations in order to keep the beauty of the equations. And in Sections \ref{5105} and  \ref{lower}, we show the persistence of invariant tori for Hamiltonian systems with rapid rotation perturbation and slow rotation perturbation, respectively. Combining Sections \ref{5105} and \ref{lower} in Section \ref{couple} we show a KAM iteration for the Hamiltonian system with multiscale rotation perturbation. In Section \ref{5108}, we sketch the proof of the third part of Theorem \ref{dingli15}. Finally, in Section \ref{application}, we show an application of our results to the weakly coupled $N-$oscillators with quasiperiodic force.

\section{Iteration Sequences}\label{5104}
Throughout the paper, unless specified explanation, we shall use the same symbol $|\cdot|$ to denote an equivalent (finite dimensional) vector norm and its induced matrix norm, absolute value of functions, and measure of sets, etc., and use $|\cdot|_D$ to denote the supremum norm of functions on a domain $D$. Also, for any two complex column vectors $\xi, \zeta$ of the same dimension, $\langle \xi ,\zeta \rangle$ always stands for $\xi^T \zeta$, i.e., the transpose of $\xi$ times $\zeta$. For the sake of brevity, we shall not specify smoothness orders for functions having obvious orders of smoothness indicated by their derivatives taking. Moreover, all Hamiltonian functions in the sequel are associated to the standard symplectic structure. All constants below are positive and independent of the iteration process and denote by $c$.

As we all known, the celebrated KAM theorem is proved by the Newton-type iteration procedure which involves an infinite sequence of coordinate changes. From each cycle of KAM steps, one can find the constructions and estimates of desired symplectic transformations and their domains, perturbed frequencies and new perturbations. To continue the KAM iteration, we need the following iteration sequences for all $\nu = 1,2,\cdots$:
\begin{align*}
 r_\nu &= r_0 (1 - \sum_{i= 1}^{\nu} \frac{1}{2^{i+1}}),\quad s_\nu = \frac{1}{8} \alpha_{\nu -1} s_{\nu -1},
\quad \alpha_{\nu} = \mu_\nu ^{2 \sigma} =\mu_\nu ^{\frac{1}{m+1}},\\
\beta_\nu &= \beta_0 (1- \sum_{i=1}^\nu \frac{1}{2^{i+1}}),\quad   \mu_\nu = 8^m c_0 \mu_{\nu -1}^{1 + \sigma },\quad  \gamma_\nu = \gamma _0 (1- \sum _{i=1} ^{\nu} \frac{1}{2^{i+1}}),\\
 K_{\nu+1}^a &=([\log \frac{1}{\mu_\nu}]+1)^{3\eta},\quad K_{\nu+1}^b = ([\frac{1}{\lambda_2}] +1)^2 ([\log \frac{1}{\mu_\nu}]+1)^{3\eta},\\
\tilde{D}_\nu &= D(r_\nu + \frac{3}{4}(r_{\nu-1}-r_\nu), \beta _\nu ), \quad D_\nu = D(r_\nu, s_\nu),~\\
\tilde{D}_\nu^d &= D^d(r_\nu + \frac{3}{4}(r_{\nu-1}-r_\nu), \beta _\nu ),\quad D_\nu^d = D^d(r_\nu, s_\nu),
\end{align*}
\begin{align*}
 G_{\nu+1}^a &= \left\{\xi \in G^a_\nu : |\langle k , \omega_\nu^a(\xi) \rangle| > \frac{\gamma}{|k|^{\tau}}, \quad for~ all ~~0<|k|\leq K_{\nu+1}^a \right\},\\
 G_{\nu+1}^b &= \left\{\xi \in G_\nu^b: |\langle k , \omega_\nu^b (\xi) \rangle| > \frac{\gamma}{|k|^{\tau}},\quad for~ all ~~0<|k|\leq K_{\nu+1}^b \right\},\\
 G_{\nu+1}^d &= \left\{\xi \in G_\nu^d: |\frac{\delta_1}{\lambda_1}\langle k_1, \omega \rangle+ \delta_2 \langle k_2, \Lambda \rangle+\delta_3 \lambda_2 \langle k_3, \Omega \rangle)| > \frac{|\tilde{\lambda}|\gamma}{|\mathbf{k}|^{\tau}},\quad for~ all ~~0<|\mathbf{k}|\leq K_{\nu+1}^b\right\},\\
 \Gamma_{\nu+1}^a &=  \sum\limits_{\substack{0<|k|\leq K_+^a,\\ |i|, |j|\leq m+4}} |k|^{(|l|+ |j|+1) \tau + |l|+ |i|+|j|+1} \lambda_1^{-|i|} e^{-\frac{  |k| (r- r_+)}{8 \lambda_1}},\\
  \Gamma_{\nu+1}^b &= \sum\limits_{\substack{0<|k|\leq K_+^b,\\ |i|, |j|\leq m+4}} |k|^{(|l|+ |j|+1) \tau + |l|+ |i|+|j|+1} {\lambda_2^{|i| -2}} e^{-\frac{  |k|\lambda_2(r -r_+)}{8}},\\
  \Gamma_{\nu+1}^d &= \sum_{\substack{0<|k_1|+|k_2|+|k_3|\leq K_+^b\\|i_1|+|i_2|+|i_3|\leq m+4\\|\kappa_1|+|\kappa_2|+|\kappa_3|\leq m+4}} \frac{|\mathbf{k}|^{(|l|+|i_1|+|i_2|+|i_3|+1)\tau+ |l|+|i_1|+|i_2|+|i_3|+1}}{\lambda_2^{\delta_3}}  |\frac{\delta_1 k_1}{\lambda_1}|^{\kappa_1}|\delta_2 k_2|^{\kappa_2} |\delta_3 \lambda_2 k_3|^{\kappa_3}\\
  & \cdot e^{-\frac{\delta_1k_1(r -r_+)}{8 \lambda_1}}
  e^{-\frac{\delta_2k_2 (r-r_+)}{8}} e^{-\frac{\lambda_2 \delta_3 k_3 (r -r_+)}{8}},
\end{align*}
where $\sigma \in (0, \frac{1}{3})$, $m>1$ are fixed, $\eta$ is a fixed positive integer such that $(1 + \sigma)^\eta > 2$ for $\sigma = \frac{1}{2(m+1)}$, $\tau > d (d -1)-1$ is given, $c_0$ is the biggest one among the constants in the iteration, $\mathbf{k} = (k_1, k_2, k_3)$, $\tilde{\lambda}= (\frac{\delta_1}{\lambda_1}, \delta_2 ,  \delta_3 \lambda_2),$ $|(x_1, x_2,x_3)| = |x_1|+ |x_2| +|x_3|$, $r_0  = r$, $\beta_0 = s$, $\gamma_0 = \varepsilon^{\frac{\frac{1}{3} - \sigma}{d + m+5}}$, $ s_0 = \varepsilon^{\frac{2}{3 m}}$, $\mu_0 = \varepsilon^\sigma$, and $\delta_i$ is a special function that satisfies $\delta_i = 0$, if $k_i = 0$ and $\delta_i = 1$, if $k_i \neq 0$, $1\leq i\leq 3$.

The proof of KAM theorems is processed by the mathematical induction, i.e. we first prove the correctness for the $0-$th step and then we show the steps from $\nu$ to $\nu+1$. For the sake of convenience, we shall omit the index for all quantities of the $\nu-$th KAM step and use $'+'$ to index all quantities in the ${(\nu+ 1)}-$th KAM step. To process our KAM step, we need the following iterative constants and iterative domains:
\begin{align*}
 r_+ &= \frac{r}{2}+ \frac{r_0}{4}, \quad s_+ = \frac{1}{8}\alpha s, \quad \alpha = \mu ^{\frac{1}{m+1}}, \quad \gamma _+ = \frac{\gamma}{2} + \frac{\gamma_0}{4},\\
 \beta _+ &= \frac{\beta}{2}+ \frac{\beta_0}{4},\quad K_+^a = ([\log \frac{1}{\mu}]+1)^{3\eta},\quad K_+^b = ([\frac{1}{\lambda_2}] +1)^2 ([\log \frac{1}{\mu}]+1)^{3\eta},\\
D(\xi) &= \left\{ y: |y|< \xi \right \},~\xi >0, \quad D^d(\xi) = \left\{ (y,\eta,I): |y|+ |\eta|+ |I|< \xi  \right\},\quad \xi >0,\\
 D_{\frac{i}{8} {\alpha}} &= D\left(r_+ +\frac{i-1}{8}(r - r_+), \frac{i}{8} \alpha s\right), \quad i= 1,\cdots ,8,\quad D_+ = D_{\frac{1}{8} {\alpha}} = D(r_+, s_+),\\
 D_{\frac{i}{8} {\alpha}}^d &= D^d\left(r_+ +\frac{i-1}{8}(r - r_+), \frac{i}{8} \alpha s\right), \quad i= 1,\cdots ,8,\quad D_+^d = D_{\frac{1}{8} {\alpha}}^d = D^d(r_+, s_+),\\
  \hat{D}(\xi) &= D\left(r_+ {+} \frac{7}{8} (r- r_+) ,\xi \right),~\xi >0,\quad \tilde{D}_+= D(r_+ + \frac{3}{4}(r-r_+),\beta_+), \\
  \hat{D}^d(\xi) &= D^d\left(r_+ {+} \frac{7}{8} (r- r_+) ,\xi \right),~\xi >0,\quad \tilde{D}_+^d= D^d(r_+ + \frac{3}{4}(r-r_+),\beta_+),
  \end{align*}
 \begin{align*}
 G_+^a &= \left\{\xi \in G^a : |\langle k , \omega^a(\xi) \rangle| > \frac{\gamma}{|k|^{\tau}}, \quad for~ all ~~0<|k|\leq K_{+}^a \right\},\\
 G_+^b &= \left\{\xi \in G^b: |\langle k , \omega^b(\xi) \rangle| > \frac{\gamma}{|k|^{\tau}},\quad for~ all ~~0<|k|\leq K_{+}^b \right\},\\
   G_+^d &= \left\{\xi \in G_\nu^d: |\frac{\delta_1}{\lambda_1}\langle k_1, \omega \rangle+ \delta_2 \langle k_2, \Lambda \rangle+\delta_3 \lambda_2 \langle k_3, \Omega \rangle)| > \frac{|\tilde{\lambda}|\gamma}{|\mathbf{k}|^{\tau}},\quad for~ all ~~0<|\mathbf{k}|\leq K_{\nu+1}^b\right\},\\
 \Gamma_+^a &= \sum_{\substack{0<|k|\leq K_+^a,\\ |i|, |j|\leq m+4}} |k|^{(|l|+ |j|+1) \tau + |l|+ |i|+|j|+1} \lambda_1^{-|i|} e^{-\frac{|k| (r - r_+)}{8\lambda_1}},\\
 \Gamma_+^b &= \sum\limits_{\substack{0<|k|\leq K_+^b,\\ |i|, |j|\leq m+4}} |k|^{(|l|+ |j|+1) \tau + |l|+ |i|+|j|+1}{\lambda_2^{|i| -2}} e^{-\frac{  |k|\lambda_2(r -r_+)}{8}},\\
\Gamma_+^d &=  \sum_{\substack{0<|k_1|+|k_2|+|k_3|\leq K_+^b\\|i_1|+|i_2|+|i_3|\leq m+4\\|\kappa_1|+|\kappa_2|+|\kappa_3|\leq m+4}} \frac{|\mathbf{k}|^{(|l|+|i_1|+|i_2|+|i_3|+1)\tau+ |l|+|i_1|+|i_2|+|i_3|+1} }{\lambda_2^{\delta_3}} |\frac{\delta_1 k_1}{\lambda_1}|^{\kappa_1}\\&~~~~~~~~~~~~~\cdot |\delta_2 k_2|^{\kappa_2} |\delta_3 \lambda_2 k_3|^{\kappa_3}e^{-\frac{\delta_1k_1(r -r_+)}{8 \lambda_1}} e^{-\frac{\delta_2k_2 (r-r_+)}{8}} e^{-\frac{\lambda_2 \delta_3 k_3 (r -r_+)}{8}}.
\end{align*}

For the simplicity  of the equations in Section \ref{couple}, we introduce the following notations:
\begin{eqnarray*}
\mathfrak{A}_\nu^d & =&
                  \left(
                    \begin{array}{ccc}
                       \frac{\partial^ 2 N_\nu^d}{\partial y^2} & \frac{\partial^2 N_\nu^d}{\partial y \partial\eta} & \frac{\partial^ 2 N_\nu^d}{\partial y \partial I} \\
                     \frac{\partial^2N_\nu^d}{\partial \eta \partial y} &   \frac{\partial^2N_\nu^d}{\partial \eta^2} & \frac{\partial^2N_\nu^d}{\partial \eta \partial I}\\
                      \frac{\partial^2 N_\nu^d}{\partial I \partial y} & \frac{\partial^2 N_\nu^d}{\partial I \partial \eta} & \frac{\partial^2 N_\nu^d}{\partial I^2} \\
                    \end{array}
                  \right),~~~
 [\mathfrak{R}_\nu] = \left(
                    \begin{array}{ccc}
                      \frac{\partial^2 [R_\nu^d]}{\partial y^2} & \frac{\partial^2 [R_\nu^d]}{\partial y \partial \eta} & \frac{\partial^2 [R_\nu^d]}{\partial y \partial I} \\
                      \frac{\partial^2 [R_\nu^d]}{\partial \eta \partial y} & \frac{\partial^2 [R_\nu^d]}{\partial \eta^2} & \frac{\partial^2 [R_\nu^d]}{\partial \eta \partial I} \\
                      \frac{\partial^2[R_\nu^d]}{\partial I \partial y} &\frac{\partial^2 [R_\nu^d]}{\partial I \partial \eta} &\frac{\partial^2 [R_\nu^d]}{\partial I^2} \\
                    \end{array}
                  \right),\\
 \hat{\mathfrak{h}}_\nu &=& \left(
                    \begin{array}{ccc}
                      \frac{\partial^2 \hat{h}_\nu^d}{\partial y^2} & \frac{\partial^2 \hat{h}_\nu^d}{\partial y \partial \eta} & \frac{\partial^2 \hat{h}_\nu^d}{\partial y \partial I} \\
                      \frac{\partial^2 \hat{h}_\nu^d}{\partial \eta \partial y} & \frac{\partial^2 \hat{h}_\nu^d}{\partial \eta^2} & \frac{\partial^2 \hat{h}_\nu^d}{\partial \eta \partial I} \\
                      \frac{\partial^2 \hat{h}_\nu^d}{\partial I \partial y} &\frac{\partial^2 \hat{h}_\nu^d}{\partial I \partial \eta} &\frac{\partial^2 \hat{h}_\nu^d}{\partial I^2} \\
                    \end{array}
                  \right), ~~
\mathbf{a}_\nu = \left(
              \begin{array}{c}
                \omega_\nu^d \\
                \Lambda_\nu^d \\
                \Omega_\nu^d \\
              \end{array}
            \right)= \left(
                       \begin{array}{c}
                         \partial_y N_\nu^d \\
                         \partial_\eta N_\nu^d \\
                         \partial_I N_\nu^d\\
                       \end{array}
                     \right)
            ,~
\\
\mathbf{b} &=& \left(
              \begin{array}{c}
                y \\
                \eta \\
                I \\
              \end{array}
            \right),
\mathbf{b_*} = \left(
  \begin{array}{c}
    y_*^d \\
    \eta_*^d \\
    I_*^d \\
  \end{array}
\right),~
\mathbf{P}_\nu= \left(
              \begin{array}{c}
                ({P_\nu})_{000100}^d \\
                ({P_\nu})_{000010}^d \\
                ({P_\nu})_{000001}^d \\
              \end{array}
            \right),
\mathbf{\partial \hat{h}_\nu} = \left(
                         \begin{array}{c}
                           \partial_y \hat{h}_\nu \\
                           \partial_\eta \hat{h}_\nu \\
                           \partial_I \hat{h}_\nu \\
                         \end{array}
                       \right),~\\
\mathbf{\partial_d \hat{h}_\nu} &=& \left(
                         \begin{array}{c}
                           \partial_{y^d} \hat{h}_\nu \\
                           \partial_{\eta^d} \hat{h}_\nu \\
                           \partial_{I^d} \hat{h}_\nu \\
                         \end{array}
                       \right),
\mathbf{\partial [R_\nu]} = \left(
                   \begin{array}{c}
                     \partial_y [R_\nu] \\
                     \partial_\eta [R_\nu] \\
                     \partial_I [R_\nu]\\
                   \end{array}
                 \right),~
\mathbf{p}_\nu = \left(
             \begin{array}{c}
               ({p_\nu})_{000100}^d \\
               ({p_\nu})_{000010}^d \\
               ({p_\nu})_{000001}^d \\
             \end{array}
           \right),\\
S \hat{h}_\nu &=& \left(
                \begin{array}{ccc}
                  \int_0^1\partial_y^2 \hat{h}_\nu^d(\theta_1 \eta)d\theta_1 & \int_0^1\partial_\eta\partial_y \hat{h}_\nu^d(\theta_1 \eta)d\theta_1 & \int_0^1\partial_I\partial_y \hat{h}_\nu^d(\theta_1 \eta)d\theta_1 \\
                  \int_0^1\partial_y \partial_\eta \hat{h}_\nu^d(\theta_1 I)d\theta_1 & \int_0^1\partial_\eta^2 \hat{h}_\nu^d(\theta_1 I)d\theta_1 & \int_0^1\partial_I \partial_\eta \hat{h}_\nu^d(\theta_1 I)d\theta_1 \\
                  \int_0^1\partial_y\partial_I \hat{h}_\nu^d(\theta_1 y)d\theta_1 & \int_0^1\partial_\eta \partial_I \hat{h}_\nu^d(\theta_1 y)d\theta_1 & \int_0^1\partial_I^2 \hat{h}_\nu^d(\theta_1 y)d\theta_1 \\
                \end{array}
              \right).
\end{eqnarray*}

\section{Rapid Rotation Case}\setcounter{equation}{0} \label{5105}
In this section, we consider a Hamiltonian system with rapid rotation perturbation, i.e., a Hamiltonian system of the following form:
\begin{eqnarray}\label{501}
H^a(x,y) &=& N^a(y) + \varepsilon P^a(\frac{x}{\lambda_1}, y),
\end{eqnarray}
which defined on the complex neighborhood $D(r,s) = \{(x,y): |{\rm Im}~x| < r, |y|<s\}$ of $\mathbb T^d \times \{0 \}  \subset\mathbb T^d \times \mathbb R^d$, where $\lambda_1= \varepsilon^{\alpha}$, $\alpha \in \mathbb R^1_+$, $N^a (y)$ is a real analytic function on a complex neighborhood of the bounded closed region $G^a$, $\varepsilon P^a(\frac{x}{\lambda_1}, y,\xi)$, a small perturbation, is a real analytic function, where $\varepsilon > 0$ is a small parameter.

Assume that

\begin{itemize}
\item[\bf{(S1)}] There exists an $N > 1$ such that
\begin{eqnarray*}
{\rm rank} \{ \partial_{y}^\alpha N^a: 1\leq |\alpha|\leq N,~~~~ \forall y \in G^a \} = d.
\end{eqnarray*}
\item[\bf{(H1)}] $A^a$ has an $n \times n$ nonsingular minor $\mathcal{A}^a$.
\end{itemize}

Then for Hamiltonian $(\ref{501})$, we have the following result.
\begin{theorem}\label{dingli9}
Let $H^a$ be analytic.  Under assumptions $\bf{(S1)}$ and $\bf{(H1)}$, there exist a $\varepsilon_0 >0 $ and a family of Cantor sets $G_\varepsilon^a \subset G^a$, $0<\varepsilon < \varepsilon_0$, such that for each $y \in G_\varepsilon^a$, the unperturbed $d-$tours $T_{y}^a$ persists and gives rise to a real analytic,  invariant $d-$torus $T_{\varepsilon, y}^a$ preserving $n$ corresponding unperturbed toral frequencies. Moreover, the relative Lebesgue measure $|G^a \setminus G_\varepsilon^a |$ tends to 0 as $\varepsilon \rightarrow 0$.
\end{theorem}

 The main task of this section is to prove Theorem \ref{dingli9} by KAM iteration. With the transformation: $y\rightarrow y +\xi, ~~x \rightarrow x,$  where $ \xi \in G^a$, Hamiltonian system (\ref{501}) reads
\begin{eqnarray}
\label{584}H^a (x,y,\xi)&=& N^a (y,\xi) + \varepsilon P^a(\frac{x}{\lambda_1},y,\xi),\\
\nonumber N^a(y,\xi) &=& e^a + \langle \omega^a(\xi), y \rangle + h^a(y,\xi),
\end{eqnarray}
with $h^a(y,\xi)= \frac{1}{2}\langle y, A^a(\xi) y\rangle + \hat{h}^a(y, \xi)$, $\hat{h}^a (y,\xi) = O(|y|^3)$, where $\omega^a(\xi) = \partial_y N^a (\xi)$, $A^a(\xi) = \partial_y ^2 N^a(\xi)$, $\lambda_1= \varepsilon^{\alpha}$, $\alpha \in\mathbb R_+^1$.

Denote by $P_0 ^a=  \varepsilon P^a(\frac{x}{\lambda_1},y,\xi)$. Then, with the Cauchy estimate, obviously,
\begin{eqnarray}\label{507}
|\partial_\xi^l P_0^a|_{D(r,s)} \leq c \gamma_0^{d+m+5} s_0^m \mu_0, \quad ~|l| < d,
\end{eqnarray}
where $c$ is a constant.

In other words, we have
\begin{eqnarray*}
H^a (x,y,\xi)&=& N_0^a (y,\xi) +  P_0^a(\frac{x}{\lambda_1},y,\xi),\\
\nonumber N^a(y,\xi) &=& e_0^a + \langle \omega_0^a(\xi), y \rangle + h_0^a(y,\xi),
\end{eqnarray*}
with $h_0^a(y,\xi)= \frac{1}{2}\langle y, A_0^a(\xi) y\rangle + \hat{h}_0^a(y, \xi)$, $\hat{h}_0^a (y,\xi) = O(|y|^3)$, where $\omega_0^a(\xi) = \partial_y N_0^a (\xi)$, $A_0^a(\xi) = \partial_y ^2 N_0^a(\xi)$, $\lambda_1= \varepsilon^{\alpha}$, $\alpha\in\mathbb R_+^1$. Moreover,
\begin{eqnarray*}
|\partial_\xi^l P_0^a|_{D(r,s)} \leq c \gamma_0^{d+m+5} s_0^m \mu_0, \quad |l| < d.
\end{eqnarray*}

\subsection{KAM step}\setcounter{equation}{0}\label{KAM step}
Now, suppose that after $\nu-$th step, we have arrived at the real analytic Hamiltonian system of the following form:
\begin{eqnarray}
\label{587}H^a (x,y,\xi)&=& N^a (y,\xi) +  P^a(\frac{x}{\lambda_1},y,\xi),\\
\nonumber N^a(y,\xi) &=& e^a + \langle \omega^a(\xi), y \rangle + h^a(y,\xi),\\
\nonumber h^a (y, \xi) &=& \frac{1}{2} \langle y, A^a(\xi) y\rangle + \hat{h}^a (y,\xi),\\
\label{588} |\partial_\xi^l P^a|_{D(r,s)} &\leq& c \gamma^{d+m+5} s^m \mu, ~~~~~~~~~|l| < d,
\end{eqnarray}
where $y\in G^a \subset\mathbb R^d,$ $x \in\mathbb T^d$, $\xi \in G^a$, $\lambda_1 = \varepsilon^{\alpha}$, $\alpha \in\mathbb R_+^1$, $\hat{h}^a (y,\xi) = O(|y|^3)$.

By considering both averaging and translation, we will find a symplectic transformation $\Phi _ {+}^a$, which, on a small phase domain $D(r_+, s_+)$  and a smaller parameter domain $G_+^a$, transforms Hamiltonian (\ref{587}) into the Hamiltonian of the next KAM step, i.e.,
 \begin{center}
  $H_{+}^a = H^a{\circ} {\Phi _+^a}= {N_+^a} +{P_+^a}$,
 \end{center}
where $N_+^a$, $P_+^a$ enjoy similar properties to $N^a$, $P^a$ respectively.

\subsubsection{{Truncation}}
Consider the Taylor - Fourier series of $P^a (\frac{x}{\lambda_1}, y, \xi)$,
 \begin{eqnarray*}
   P^a(\frac{x}{\lambda_1}, y, \xi) = \sum _{|k| \in Z^d,~ |\jmath|\in Z_{+}^d} P_{k\jmath}^a {y^\jmath} {e^{\sqrt{-1} \frac{\langle k, x\rangle}{\lambda_1}} },
\end{eqnarray*}
and let $R^a(\frac{x}{\lambda_1}, y, \xi)$ be the truncation of $P^a(\frac{x}{\lambda_1}, y, \xi)$ of the form:
\begin{eqnarray}
\nonumber    R^a(\frac{x}{\lambda_1}, y, \xi) = \sum _{|k| \leq K_+^a,~ |\jmath| \leq m} P_{k\jmath}^a {y^\jmath} {e^{\sqrt{-1} \frac{\langle k, x\rangle}{\lambda_1}} }.
\end{eqnarray}

Standardly, with the help of the Cauchy estimate and the following assumption
\begin{eqnarray}\label{5a1}
 {\int _{K_+^a}^{\infty}} {t^{d}} {e^{-t \frac{r- r_+}{4\lambda_1}}} dt \leq \mu
\end{eqnarray}
 \noindent on $D_{{\frac{7}{8}} \alpha},$ we have
 \begin{eqnarray}
  \nonumber |\partial_{\xi}^{l}P^a- \partial_{\xi}^{l}R^a|_{D_{\frac{7}{8} {\alpha}}} &\leq& c {\gamma}^{d+ m+5} {s^m} {\mu }^2,
 \end{eqnarray}
 and
 \begin{eqnarray}
   \nonumber |\partial_{\xi}^{l} R^a |_{D_{{\frac{7}{8} {\alpha}}}} &\leq&  c  {\gamma} ^{d+ m +5} {s^m} {\mu}.
 \end{eqnarray}
The details can be obtained with the same techniques as ones in \cite{LLL}.

\subsubsection{Homology Equation}\label{homology equation} \label{5135}
To process the KAM iteration, the most important thing is the invariance of the Hamiltonian in form, which holds by the following homology equation: \begin{eqnarray}\label{5Posion1}
  \{N^a,F^a\}+ R^a-[R^a]=0,
\end{eqnarray}
where \begin{eqnarray}\label{508}
F^a = \sum _{0 < |k| \leq K_+^a,~ |\jmath| \leq m} f_{k\jmath}^a {y^\jmath} {e^{\sqrt{-1} \frac{\langle k, x\rangle}{\lambda_1}} },
 \end{eqnarray}
 and $[R^a]= \int _{\mathbb T^d} R^a(\frac{x}{\lambda_1},y,\xi)dx$ is the average of truncation $R^a$, and $\{\cdot,\cdot\}$ represents Poisson brackets.

In view of (\ref{5Posion1}), comparing coefficients, we have
\begin{eqnarray} \label{509}
\frac{\sqrt{-1}}{\lambda_1} \langle k, \omega^a+ \partial_y h^a\rangle f_{k\jmath}^a = P_{k\jmath}^a.
\end{eqnarray}
Denote $M^{*a} = \max \limits_{|l|\leq d, |j|< m+5, |y|\leq \beta_0} |\partial_\xi^l \partial_y^j h_0^a (y,\xi)|$. With the assumptions
\begin{eqnarray}
\label{5b1} \max\limits_{|l|\leq d, |j|< m+5} |\partial_\xi^l \partial_y^j h^a - \partial_\xi^l \partial_y^j h_0^a|_{D(s)\times G_+^a} &\leq& \mu_0^{\frac{1}{2}},\\
\label{5c1} \frac{\gamma - \gamma _+}{(M^{*a} +1) {K_+^a}^{\tau + 1}}&>& 2s,
\end{eqnarray}
we have $|\partial_y h^a(y)|  \leq (M^{*a} +1)s\leq\frac{\gamma}{2 |k|^{\tau+1}}.$ Hence, on $G_+^a$,
\begin{eqnarray}\label{560}
|L_k^a| = |\frac{\sqrt{-1}}{\lambda_1} \langle k, \omega^a+ \partial_y h_0^a\rangle| \geq \frac{\gamma}{2 \lambda_1 |k|^\tau }.
\end{eqnarray}
Recalling differential and integral calculus and using (\ref{5c1}) and (\ref{560}), inductively, we deduce that
\begin{eqnarray}
\nonumber  |\partial_\xi^l \partial_y^j {L_k^a}^{-1}|_{D(s)\times G_+^a} &\leq& c \frac{1}{\lambda_1^{|j|+|l|}}|k|^{|j|+|l|} |{L_k^a}^{-1}|^{|l|+|j| +1}\\
\label{511} &\leq& \frac{c |k|^{(|j|+ |l|+1)\tau +|j|+|l|}{\lambda_1}}{\gamma^{|j|+ |l|+1}}.
\end{eqnarray}
Therefore, with $|\partial_\xi^l P_{k\jmath}^a|_{G_+^a} < \gamma^{d+m+5} s^{m - |\jmath|} \mu e^{\frac{-|k| r}{\lambda_1}}$, ~$|\jmath| \leq m,$ we have
\begin{eqnarray}\label{512}
 |\partial_\xi^l \partial_y^j f_{k\jmath}^a| \leq \left\{
                     \begin{array}{ll}
                       c |k|^{(|j|+ |l|+1)\tau +|j|+|l|+1} s^{m - |j|} \lambda_1 e^{- \frac{|k|r}{\lambda_1}} \mu, & \hbox{$|j|\leq m$;} \\
                       c |k|^{(|j|+ |l|+1)\tau +|j|+|l|+1} \lambda_1 e^{- \frac{|k|r}{\lambda_1}} \mu, & \hbox{$m<|j|\leq m+4$}
                     \end{array}
                   \right.
\end{eqnarray}
for $(y,\xi) \in D(s) \times G_+^a$, $0 < |k|\leq K_+^a$, $|l| \leq d$.

Combining (\ref{508}) and (\ref{512}), yields
\begin{eqnarray}
\nonumber |\partial_\xi^l \partial_x^i \partial_y^j F^a| &=& |\sum_{0<|k|\leq K_+^a, |\jmath|\leq m}(\frac{\sqrt{-1} k}{\lambda_1})^i \partial_y^j(\partial_\xi^l f_{k\jmath}^a y^\jmath) e^{\sqrt{-1} \frac{ \langle k,x\rangle}{\lambda_1}}|\\
\label{517} &\leq& \left\{
                     \begin{array}{ll}
                       c s^{m - |j|} \mu \Gamma^a(r - r_+), & \hbox{$|j| \leq m$;} \\
                       c \mu \Gamma^a(r - r_+), & \hbox{$m<|j|\leq m +4$,}
                     \end{array}
                   \right.
\end{eqnarray}
where $\Gamma^a(r - r_+) = \sum\limits_{\substack{0<|k|\leq K_+^a,\\ |i|, |j|\leq m+4}} |k|^{(|l|+ |j|+1) \tau + |l|+ |i|+|j|+1} \lambda_1^{-|i|} e^{-\frac{  |k| (r- r_+)}{8 \lambda_1}}$.
\subsubsection{Frequency Retention}\label{Frequency Retention}
Under the time $1-$map $\Phi _{F^a}^1$ of the flow generated by a Hamiltonian $F^a$, with (\ref{5Posion1}) we have
\begin{eqnarray*}
  \bar{H}_+^a &=& H^a \circ \Phi _{F^a}^1 = (N^a+ R^a)\circ \Phi _{F^a}^1 + (P^a- R^a)\circ \Phi _{F^a}^1 \\
  &=& N^a + \{N^a, F^a\}+ \int_0^1 (1 - t)\{\{N^a, F^a\}, F^a\}\circ \Phi_{F^a}^t dt+ R^a \\
  &~&+ \int_0^1\{R^a, F^a\}\circ \Phi_{F^a}^t dt + (P^a-R^a)\circ\Phi_{F^a}^1\\
   &=&N^a + [R^a] + \bar{P}_+^a ( \frac{x}{\lambda_1},y, \xi),
\end{eqnarray*}
where
\[
 \bar{P}_+^a( \frac{x}{\lambda_1},y, \xi) = \int_0^1 \{R_t^a, F^a\} \circ \Phi_{F^a}^t dt+ (P^a- R^a)\circ \Phi_{F^a}^1,
 \quad R_t^a = t R ^a+(1 - t) [R^a].
 \]
To eliminate the frequency drift, we consider the transformation $\phi^a : x \rightarrow x, ~y \rightarrow y+ y_{*}^a.$ Then
\begin{eqnarray*}
H_+^a &=& \bar{H}_+^a \circ \phi = e^a + \langle \omega^a, y+y_*^a\rangle + \langle y+y_*^a, A^a (y+y_*^a)\rangle + \hat{h}^a (y+y_*^a, \xi) \\
&~&~~~~~~+ [R^a](y + y_*^a) +\bar{P}_+^a ( \frac{x}{\lambda_1},y, \xi) \circ \phi^a\\
 &=& e + \langle\omega^a, y_*^a\rangle+ \langle y_*^a, A^a y_*^a\rangle + \hat{h}^a(y_*^a) + [R^a](y_*^a)+\langle\omega^a,y\rangle+ \frac{1}{2}\langle y, A^a y_*^a\rangle\\
  &~&~ + \langle\partial_y \hat{h}^a(y_*^a), y\rangle + \langle P_{01}^a, y\rangle+\frac{1}{2} \langle y, A^ay\rangle + \frac{1}{2} \langle y, \partial_y^2 h^a(y_*^a) y\rangle \\
  &~&~+ \frac{1}{2} \langle \partial_y^2 [R^a](y_*^a)y, y\rangle+\hat{h}^a(y+y_*^a, \xi) - \hat{h}^a(y_*^a) - \langle \partial_y \hat{h}^a(y_*^a), y\rangle \\
  &~&~- \frac{1}{2} \langle y, \partial_y^2 \hat{h}^a (y_*^a) y\rangle+[R^a](y+y_*^a) - [R^a](y_*^a) - \langle \partial_y [R^a](y_*^a), y\rangle \\
  &~&~- \frac{1}{2}\langle \partial_y^2[R^a](y_*^a)y, y\rangle+\langle \partial_y [R^a](y_*^a), y\rangle  - \langle P_{01}^a, y\rangle +\bar{P}_+^a ( \frac{x}{\lambda_1},y, \xi) \circ \phi^a.
\end{eqnarray*}

Let $\mathbf{y}^a$ and $p_{01}^a$ be the vectors formed by the $n$ components of $y$ and $P_{01}^a$, respectively, and denote $\hat{h}^a(\mathbf{y}^a)= \hat{h}^a((\mathbf{y}^a,0)^T).$ Then by the implicit function theorem, the equation
\begin{eqnarray}
\label{514} \mathcal{A}^a\mathbf{y}^a + {\partial_{\mathbf{y}^a}}\hat{h}^a(\mathbf{y}^a) &=& - p_{01}^a
\end{eqnarray}
admits a unique solution $\mathbf{y}_{*}^a$ on $D(s)$, which also smoothly depends on $\xi$, where $\mathcal{A}^a$ is an $n\times n$ nonsingular minor of $A^a$. Define $ y_{*}^a= (\mathbf{y}_{*}^a,0 )^T$, by (\ref{514}), we clearly have
 \begin{eqnarray*}
     A^a y_{*}^a + \partial _{y} {\hat{h}^a(y_{*}^a)} = -(p_{01}^a,0 )^T.
   \end{eqnarray*}
 Then
 \begin{eqnarray*}
    H_+^a = \bar{H}_{+}^a \circ \phi^a =  N_+^a + P_+^a= e_{+}^a + \langle \omega_+^a , y\rangle + h_{+}^a(y)+P_+^a,
 \end{eqnarray*}
where
\begin{eqnarray}
 \nonumber e_+^a &=& {e}^a + \langle \omega^a, y_{*}^a\rangle +\frac{1}{2} \langle y_{*}^a ,A^ay_{*}^a\rangle + \hat{h}^a(y_{*}^a)+ [R^a](y_*^a),~~ \omega_+^a = \omega^a +  P_{01}^a -\left(
                                                   \begin{array}{c}
                                                      p_{01}^a \\
                                                     0 \\
                                                   \end{array}
                                                 \right),\\
\nonumber A_+^a &=& A ^a+ \partial_y^2 \hat{h}^a(y_*^a) + \partial_y^2 [R^a](y_*^a),~~~~~~ h_+^a (y) = \frac{1}{2}\langle y, A_+^a y\rangle + \hat{h}_+^a (y),\\
\nonumber \hat{h}_+^a&=& \hat{h}^a (y+ y_{*}^a)- \hat{h}^a(y_{*}^a)- \langle\partial_{y}\hat{h}^a(y_{*}^a) , y \rangle - \frac{1}{2} \langle y, \partial_{y}^2\hat{h}^a(y_{*}^a)y \rangle + [R^a](y+y_*^a) \\
\nonumber &~&~~~ - [R^a](y_*^a)- \langle \partial_y[R^a](y_*^a), y\rangle- \frac{1}{2}\langle y, \partial_y^2 [R^a](y_*^a) y\rangle,\\
\nonumber P_+^a &=& \bar{P}_+ ^a(\frac{x}{\lambda_1},y,\xi) \circ \phi^a + \psi^a,~~~~\psi^a = \langle \partial_y [R^a](y_*^a), y\rangle  - \langle P_{01}^a, y\rangle.
\end{eqnarray}

\subsubsection{ Estimate on $N_+^a$}\label{Estimation N_+}
Denote $M_{*}^a = \max\limits_{\xi \in G_0^a} |{\mathcal{A}_0^a}^{-1}(\xi)| +1$ and let $\mu_0$ small enough, say, $\mu_0 < \frac{1}{8{M_*^a}^2(M^{*a} +1)}$, such that $M_*^a(M^{*a} +1)s_0^2 < \frac{1}{4}$.
For $\xi \in G_+^a$, we denote
\begin{eqnarray*}
B^a(y,\xi) = \mathcal{A}^a + \int _0^1 \partial_y^2 \hat{h}^a(\theta y) d\theta.
\end{eqnarray*}
 Then by (\ref{514}),
 \begin{eqnarray}\label{516}
  B^a(\mathbf{y}_{*}^a)\mathbf{y}_{*}^a = - p_{01}^a.
\end{eqnarray}
With assumption (\ref{5b1}) and using the same method in \cite{Li}, we can get that $ B^a(\mathbf{y}_{*}^a)$ is nonsingular and $|{B^a}^{-1}(\mathbf{y}_{*}^a)| \leq \frac{|{\mathcal{A}_0^a}^{-1}|}{I- |{\mathcal{A}_0^a} - B^a(\mathbf{y}_{*}^a)||{\mathcal{A}_0^a}^{-1}|} \leq 2M^{*a}$. Hence,
\begin{center}
$|y_{*}^a| =|\mathbf{y}_{*}^a|\leq 2 M_{*}^a |\partial _y P^a|_{D(s)} \leq 2 M_{*}^a \gamma ^{d + m + 5} s^{m-1} \mu $.
\end{center}
Differentiating (\ref{516}) with respect to $\xi$ and by induction, we have
\begin{eqnarray*}
|\partial _{\xi}^l y_{*}^a|< c M_{*}^a\gamma^{d + m + 5} s^{m - 1} \mu,
\end{eqnarray*}
 provided
\begin{eqnarray}\label{5d1}
    4  M_{*}^a (M^{*a}+1)\gamma ^{d+ m+5}s^{m-1} \mu < \frac{1}{2}.
\end{eqnarray}

Hence
 \begin{eqnarray}
 \nonumber |\partial_{\xi}^{l}{e_+^a} - \partial_{\xi}^{l}e^a| _{G_+^a} &\leq& c \gamma ^{d+m+5} s^{m-1} \mu,\\
  \nonumber |\partial_{\xi}^{l}{\omega_+^a} - \partial_{\xi}^{l}{\omega^a}| _{G_+^a} &\leq& c \gamma ^{d+m+5} s^{m-1} \mu,\\
 \label{580a} |\partial_y ^j \partial_{\xi}^{l}{h_+^a} - \partial_y ^j \partial_{\xi}^{l}{h^a}| _{G_+^a} &\leq& \left\{
                                                                                                                       \begin{array}{ll}
                                                                                                                         c \gamma ^{d+m+5} s^{m - |j|}\mu, & \hbox{$|j| \leq m$;} \\
                                                                                                                         c \gamma ^{d+m+5}\mu, & \hbox{$m<|j| \leq m+4$.}
                                                                                                                       \end{array}
                                                                                                                     \right.
 \end{eqnarray}

\subsubsection{Estimate on $\Phi_+^a$}
Denote $\Phi_+^a = \Phi_{F^a}^1 \circ \phi^a$ and
\begin{eqnarray}\label{519}
   \Phi _{F^a}^t = id + \int _0 ^t X_{F^a} \circ \Phi_{F^a} ^{u} d u,
\end{eqnarray}
where $X_{F^a} = (F_y^a, -F_x^a)^T$ denotes the vector field generated by $F^a$, the estimates of $\Phi_+^a$ are intimately tied to one of $X_{F^a}$, implying the essentiality for the estimates of $X_{F^a}$.

By (\ref{517}), $|\partial _\xi ^l F^a |_{D_{\frac{7}{8}\alpha}} \leq c s^m \mu \Gamma^a(r -r_+),$ and thus, by the Cauchy estimate on ${D_{\frac{3}{4}\alpha}}$,
\begin{eqnarray}\label{520}
 (r - r_+) |\partial _{\xi}^l \partial _y F^a|,~s| \partial _{\xi}^l \partial _x F^a|  ~\leq ~ c s^{m} \mu \Gamma^a(r -r_+).
\end{eqnarray}
Then we have, inductively, $|D^n \partial _\xi ^l F^a| \leq c  \mu \Gamma^a(r -r_+),$ $n \leq 4.$

Denote $\Phi _{F^a}^t =(\phi_1^a,\phi_2^a)^T$, where $\phi_1^a$, $\phi_2^a$ are components of $\Phi_{F^a}^t$ in the directions of $x,$ $y,$ respectively. Let $(x, y)$ be any point in $D_{\frac{1}{4} \alpha}$ and let $t_{*} = sup \{ t\in [0,1] : \Phi_{F^a}^t(x,y )\in D_\alpha \}$. Note that $D_\alpha \subset \hat{D}(s).$
By (\ref{519}), we have
\begin{center}
  $|\phi_1^a(x, y) - x| \leq \int_0^t |F_y ^a\circ \Phi_{F^a}^{u}|_{D_\alpha} du \leq|F_y^a| _{\hat{D}(s)}< c \Gamma^a(r-r_+) \mu< \frac{1}{8} (r - r_+),$
\end{center}
provided
\begin{eqnarray}\label{5f1}
   c \Gamma^a(r-r_+) \mu< \frac{1}{8} (r - r_+);
\end{eqnarray}
\begin{eqnarray*}
  |\phi_2^a (x, y) - y| \leq \int_0^t |F_x^a \circ \Phi_{F^a}^{u}|_{D_\alpha} du  \leq|F_x^a| _{\hat{D}(s)}< c \mu s^{m-1} \Gamma^a (r - r_+) \leq \frac{1}{8} \alpha,
\end{eqnarray*}
and provided
\begin{eqnarray}\label{5g1}
    c \mu s^{m} \Gamma^a (r - r_+) < \frac{1}{8} \alpha s.
\end{eqnarray}
Then $|\phi_1^a|< r_+ + \frac{3}{8} ( r- r_+)$ and $|\phi_2^a| < \frac{3}{8} \alpha s.$ Therefore, $\Phi _{F^a}^t : D_{\frac{1}{4} \alpha} \rightarrow D_{\frac{1}{2} \alpha} \subset D_\alpha$. Hence $t_* = 1$ and $\Phi_{F^a}^t: D_{\frac{\alpha}{4}} \rightarrow D_{\frac{\alpha}{2}}$. By the estimates of $|\partial _\xi ^l {y_{*}^a}|_{G _+^a}$, under assumption
\begin{eqnarray}\label{5e1}
c s^{m-1} \mu < \frac{1}{8} \alpha s,
\end{eqnarray}
it is easy to see $\phi^a : D_{\frac{1}{8} \alpha} \rightarrow D_{\frac{1}{4}\alpha}.$

With the standard Whitney extension theorem (see {\cite{Poschel,Stein}}), it is easy to see that $F^a$ and $y_{*}^a$ can be extended to functions of H\"{o}lder class $C^{m+3,d-1+\sigma_0}(\hat{D}(\beta_0)\times {G_0^a})$, respectively, where $0< \sigma_0 <1$ is fixed. Moreover,
\begin{align*}
 \|F^a\|_{{C^{m+3, d-1+\sigma_0}}(\hat{D}(\beta_0) \times G_0^a)} &\leq c \mu \Gamma^a(r - r_+),\\
 \|y_{*}^a\| _{C^{d-1 + \sigma_0 }(G_0^a)}&\leq c \mu \Gamma^a(r- r_+).
\end{align*}

The above imply that $\Phi_+^a: D_+ \rightarrow D_{\frac{1}{2}\alpha}$ is well defined, symplectic and real analytic for all $ \xi \in G_+^a$. We, now, consider $\Phi_+^a$ on the domain $\tilde{D}_+$.

It is easy to see that $\Phi_+^a$ maps $\hat{D}_+$ into $D(r,\beta)$ for all $\xi \in G_0^a$. We note that
\begin{center}
$\Phi _{F^a}^t = id + \int_0^t X_{F^a} \circ \Phi_{F^a}^u du$,~~~~~~~~~~$0\leq t \leq 1$,\\
$\|X_{F^a}\| _{C^{m+2, d-1+\sigma_0}(\hat{D}(\beta _0)\times G_0^a)} \leq c \|F^a\|_{C^{m+3, d-1+\sigma_0}(\hat{D}(\beta _0)\times G_0^a)}.$
\end{center}
Supposing
\begin{eqnarray}
\label{5h1}  c \mu \Gamma^a (r- r_+) &<& \frac{1}{8} (r - r_+),\\
\label{5i1} c \mu \Gamma ^a(r- r_+) &<& \beta - \beta_+,
\end{eqnarray}
and applying the Gronwall inequality and the definition of $\Phi_{F^a}^t$, inductively, we have that on $\tilde{D}_+ \times G _0^a$,
\begin{eqnarray}\label{521}
  |\Phi_{F^a}^t - id |, |\partial _y \Phi _{F^a}^t - I_{2d}|, |\partial_y ^j \Phi _{F^a}^t| \leq c\mu \Gamma^a (r-r_+).
\end{eqnarray}
With $
    \Phi_+^a - id = (\Phi_{F^a}^1 - id )\circ \phi^a +\left(
                                             \begin{array}{c}
                                               0 \\
                                               y_{*}^a
                                             \end{array}
                                           \right)
$, it is straightforward that $ \Phi _+^a = \Phi_{F^a}^1 \circ \phi^a : \hat{D}_+ \rightarrow D(r, \beta)$ is of classes $C^{m+2}$ and also depends $C ^{d-1+ \sigma_0}$ smoothly on $\xi \in G_0^a$, where $\sigma_0 $ described as above. Moreover,
\begin{center}
$\| \Phi_+^a - id\|_{C^{m+2,d-1+\sigma_0}(\tilde{D}_+ \times G_0^a)}\leq c \mu \Gamma^a(r- r_+).$
\end{center}

 Then under the symplectic transformation $\Phi _+^a = \Phi_{F^a} ^ 1\circ \phi^a,$ the new Hamiltonian reads
\begin{center}
   $ H^a\circ \Phi _+^a = N_+^a +P_+^a$,
\end{center}
where
\begin{eqnarray*}
    N_+^a &=& e_+^a + \langle \omega _+^a, y\rangle + h_+^a(y),~~~~~~P_+ ^a= \bar{P}_+^a \circ \phi^a  + \psi^a,\\
    \omega _+ ^a&=& \omega^a + P_{01}^a + {A^a} y_{*}^a +\partial _y \hat{h}^a(y_{*}^a),
\end{eqnarray*}
and $h_+ ^a(y)$, ${A^a}$, $\psi^a$ and $\hat{h}^a(y)$ have the same forms as above. Thus, the new normal form is reduced to the desired case.

With the assumption mentioned above we have $c s^{m+1} {\mu K_+^a}^{\tau+1}< \gamma - \gamma_+,$ then
\begin{eqnarray*}
  |\langle k , \omega_+^a  \rangle | > \frac{\gamma_+ }{ |k|^{\tau} } ~ for ~0< |k| \leq K_+^a,~\xi \in G_+^a.
\end{eqnarray*}

\subsubsection{Estimate on $P_+^a$}
 We know
\begin{eqnarray*}
P_+^a(\frac{x}{\lambda_1}, y, \xi) = \bar{P}_+^a \circ \phi^a+ \psi^a = (\int_0^1 \{R_t^a , F^a\}\circ \Phi _{F^a}^t dt + (P^a-R^a)\circ \Phi _{F^a}^1 )\circ \phi^a + \psi^a.
\end{eqnarray*}
By above estimates, we see that, for all $|l| \leq d$, $0\leq t \leq 1$,
\begin{eqnarray}
   \nonumber  |\partial _\xi^ l \{R_t^a , F^a\}\circ \Phi_{F^a}^t|_{{D_{\frac{1}{4}\alpha}}\times G_+^a }&\leq& c  \gamma^{d+m+5} s^{m} {\mu}^2 \Gamma^a (r- r_+),\\
   \nonumber |\partial _\xi^l(P^a-R^a) \circ \Phi_{F^a}^1|_{{D_{\frac{1}{4}\alpha}}\times G_+^a }&\leq&c \gamma^{d+m+5} s^m {\mu}^2,\\
   \nonumber|\partial_\xi^l \psi^a|~~~~~&\leq& c \gamma^{d+m+5} s^{m} {\mu}^2.
\end{eqnarray}
And by the estimate of $y_*^a$, we have
\begin{eqnarray*}
|\partial_\xi^l \phi^a |_{D_+ \times G_+^a} \leq c \gamma^{d +m +5} s^{m-1} \mu ~~for~~|l| \leq d.
\end{eqnarray*}
 Hence, by the definition of $P_+^a$,
\begin{center}
  $|\partial _\xi^l P_+^a |_{D_+ \times G _+^a} \leq c\gamma^{d+m+5} s^m {\mu}^2 (\Gamma^a (r- r_+) +2 )$, ~~~~~~$|l|\leq d$.
\end{center}

Let $c_0$ be the maximal one of the $c's$ mentioned above and define $\mu_+ = 8^m c_0 \mu^{1 + \sigma}.$ With the assumption
\begin{eqnarray}\label{5j1}
 \mu^{\sigma } (\Gamma ^a(r- r_+) +2) \leq \frac{\gamma_{+}^{d+m+5}}{\gamma^{d+m+5}},
 ~~~on~ D_{+} \times G_+^a,
\end{eqnarray}
we have
\begin{eqnarray}
\nonumber |\partial _{\xi}^l P_+^a| &\leq& 8^m c_0 s^2 {\mu}^{1+ \sigma} {\mu}^{\frac{1}{3}-2\sigma} {\mu}^{\sigma} {\gamma}^{d+m+5} (\Gamma^a(r- r_+) +2)\\
 \nonumber &\leq& c_0\gamma_{+}^{d+m+5} s_{+}^m \mu _+,~~~~~~~~|l|\leq d.
\end{eqnarray}

We now complete one KAM step.

\subsection{Iteration Lemma}
Consider (\ref{584}) and let $r_0$, $s_0$, $\gamma _0$, $\beta_0$, $\mu_0$, $N_0^a$, $e_0^a$, $\omega_0^a$, $h_0^a$, $A_0^a$, $\hat{h}_0^a$, $P_0^a$ be given as before. We have the following iteration lemma.
\begin{lemma}\label{lemma1}
If (\ref{588}) holds for a sufficiently small $\mu= \mu (r, s, d, \tau)$, then the KAM step described in subsection {\bf{\ref{KAM step}}} is valid for all $\nu =0,1,\cdots,$ and sequences
$G_\nu^a, H_\nu^a, N_\nu^a, e_\nu^a, \omega_\nu^a, h_\nu^a, A_\nu^a, \hat{h}_\nu^a, P_\nu^a, \Phi_\nu^a$, $\nu = 1,2,\cdots ,$ possess the following properties:
\begin{itemize}
\item[{\bf (1)}]$\Phi _\nu^a : \hat{D}\times G_0^a \rightarrow \hat{ D}_{\nu-1}$, $D_\nu \times G_\nu^a \rightarrow D_{\nu-1}$ is symplectic for each $\xi\in G_0^a$
 or $G_\nu^a$, and is of class $C^{m+2,d-1+\sigma_0}$, $C^{\iota, d}$, respectively, where $\iota $ stands for real analyticity, $0<\sigma_0<1$ is fixed, and
 \begin{eqnarray}\label{522}
 \|\Phi_\nu^a - id\| _{C^{m+2,d-1+\sigma_0}(\hat{D}_\nu \times G^a)} \leq \frac{\mu^{\frac{1}{2}}}{2^{\nu}}.
 \end{eqnarray}
Moreover, on $\hat{D}_\nu \times G_\nu^a$,
\begin{center}
  $H_\nu ^a= H_{\nu-1}^a\circ \Phi_{\nu}^a= N_{\nu}^a+ P_\nu^a,$
  \end{center}
  where $N_\nu^a = e_\nu^a + \langle \omega_\nu^a ,y\rangle + \frac{1}{2} \langle y, A_\nu^a y\rangle + \hat{h}^a(y),$ $A_\nu^a$ has an $n \times n $ nonsingular minor $\mathcal{A}_{\nu}^a$, which is nonsingular on $G_{\nu}^a$, $ \hat{h}^a(y)= O (|y|^3)$;
\item[{\bf (2)}]Under assumption $\bf{(H1)}$, we have $(\omega_\nu^a (\xi))_q=(\omega_{\nu-1}^a(\xi))_q $, $\forall$ $\xi \in G_\nu^a,$ $q= 1,2,\cdots n;$
\item[{\bf(3)}]For all $|l|\leq d ,$
  \begin{eqnarray}
   \label{523} |\partial_\xi^l e_\nu^a -\partial_\xi^l e_{\nu -1}^a | _{G_\nu^a} &\leq & \gamma _0^{d+m+4} \frac{\mu}{2^{\nu}};\\
  \label{524}  |\partial_\xi^l e_\nu^a -\partial_\xi^l e_0^a | _{G_\nu^a} &\leq&  \gamma _0^{d+m+4} \mu;\\
  \label{525} |\partial_\xi^l {\omega}_{\nu}^a -\partial_\xi^l {\omega}_{\nu -1}^a |_{G_\nu^a} &\leq&  \gamma _0^{d+m+4} \frac{\mu}{2^{\nu}};\\
   \label{526} |\partial_\xi^l {\omega}_\nu^a -\partial_\xi^l {\omega}_0^a |_{G_\nu^a} &\leq&  \gamma _0^{d+m+4} \mu;\\
 \label{527} |\partial _\xi^l h_\nu^a  -\partial _\xi^l  h_{\nu-1}^a |_{D_\nu \times G_\nu^a} &\leq&  \gamma _0 ^{d+m+4} \frac{{\mu}^{\frac{1}{2}}}{2^{\nu}};\\
 \label{528}  |\partial _\xi^l h_\nu^a  -\partial _\xi^l h_0^a |_{D_\nu \times G_\nu^a} &\leq&  \gamma _0 ^{d+m+4} {\mu}^{\frac{1}{2}};\\
  \label{529}  |\partial_{\xi}^l P_\nu^a|_{D_\nu \times G_\nu^a}&\leq&  \gamma_ \nu ^{d+m+5} s_\nu^2 {\mu_\nu};
 \end{eqnarray}
 \item[{\bf(4)}] $G_\nu^a = \{ \xi \in G_{\nu -1}^a: |\langle k, \omega _{\nu-1}^a(\xi)\rangle|> \frac{\gamma_{\nu-1}}{|k|^\tau},  ~for ~all~0< |k| \leq K_\nu^a\}.$
\end{itemize}
\end{lemma}
\begin{proof}
Actually, it suffices to verify the assumptions that we put forward above for all $\nu$. For simplicity, we let $r_0 = \beta_0= 1$. By choosing $\mu_0$ small, we also see that others assumptions are hold for $\nu =0 $.

By the definition of $\mu_\nu$, we have that $\mu_\nu = (8^m c_0)^{\frac{{(1+ \sigma)}^{\nu} -1}{\sigma}} {\mu_0}^{(1+ \sigma)^{\nu}}.$ Therefore,
\begin{eqnarray}\label{531}
    \mu_{\nu} = 8^m c_0 \mu_{\nu -1}^{1+ \sigma} < \cdots < \frac{1}{{\zeta}^{\nu}} \mu_0,
\end{eqnarray}
where $\zeta \gg1$ and $\mu _0 < (\frac{1}{8^m c_0 \zeta})^{\sigma }\ll 1.$ Then assumption (\ref{5d1}) holds. With the definition of $s_\nu$, we have
 \begin{eqnarray}\label{561}
 s_\nu= {(\frac{1}{8})}^\nu (8^m c_0)^{\frac{(1 + \sigma)^\nu - (1+\sigma) -  \sigma\nu}{\sigma^2} \frac{1}{m + 1}} \mu_0^{\frac{(1+\sigma)^\nu - 1}{\sigma(m+1)}}s_0,
 \end{eqnarray}
 and thus (\ref{5e1}) is obvious. Since
\begin{eqnarray}
 \nonumber \Gamma _\nu^a &\leq& \int _1^{\infty} {t}^{(|l| +|j|+ 1)\tau + |l| +|i| + |j| +2}  \lambda_1^{ - |i|} e^{-\frac{t(r - r_+)}{8\lambda_1}} dt\\
 \nonumber  &\leq&  \lambda_1^{- |i|} ( \lambda_1 2^{\nu+6} e^{-\frac{1}{ \lambda_1 2^{\nu+6}}}\\
 \nonumber &~&~~+  \lambda_1^2 2^{2(\nu+6)} ((|l| + |j|+ 1)\tau + |j|+ |l| +|i|+2) e^{-\frac{1}{ \lambda_1 2^{\nu+6}}}\\
\nonumber &~&~~+ \cdots +  \lambda_1^{(|l| + |j|+ 1)\tau + |j|+ |l|+ |i|+2} 2^{(\nu+6)((|l| + |j|+ 1)\tau + |j|+ |l|+ |i|+2)}\\
 \nonumber &~&~~~~~~~~~((|l| + |j|+ 1)\tau + |j|+ |l| + |i|+2)! e^{-\frac{1}{ \lambda_1 2^{\nu+6}}} )\\
\nonumber &\leq&~~  c\lambda_1^{- |i| +1} 2^{\nu+6} e^{-\frac{1}{ \lambda_1(2^{\nu+6})}},
\end{eqnarray}
where $\Gamma_{\nu}^a = \Gamma_\nu^a(r_\nu - r_{\nu-1})$,
it is clear that
\begin{center}
  $\mu_{\nu}^{\sigma} {\Gamma_\nu^a} <\mu_{\nu}^{\sigma}(\Gamma_{\nu}^a +2)< \frac{1}{\zeta^{\nu \sigma}} \mu_0^\sigma ( \lambda_1^{- |i|+1} 2^{\nu+6} e^{-\frac{1}{ \lambda_1 2^{\nu+6}}} +2) < \frac{\gamma_{\nu+1}^{d+m+5}}{\gamma_{\nu}^{d+m+5}}$,
\end{center}
where we use the fact that $\lambda_1^{- |i|+1} e^{-\frac{1}{ \lambda_1 2^{\nu+6}}} \leq e^{(i - 1) (ln (i-1) 2^{\nu +6} - 1)}$.

 Therefore, (\ref{5j1}) holds for all $\nu \geq 1$. Furthermore, assumptions (\ref{5f1}), (\ref{5h1}) and (\ref{5i1}) obviously hold.

Due to
\begin{eqnarray*}
\mu_0^{\frac{(1 + \sigma)^\nu -1}{\sigma} \frac{m}{m+1} + (1 + \sigma)^\nu} < \mu_0^{\frac{(1 + \sigma)^{\nu+1} - 1}{\sigma (m+1)}}
\end{eqnarray*}
 for $\sigma = \frac{1}{2(m+1)}$, (\ref{5g1}) holds.
Next we are going to testify assumption (\ref{5a1}). Actually, we know
\begin{eqnarray*}
\int_{K_+^a}^\infty t^n e^{- t \frac{r - r_+}{4  \lambda_1}} dt &\leq& \frac{4  \lambda_1}{r - r_+} {K_+^a}^n e^{- \frac{K_+^a(r - r_+)}{4  \lambda_1}} + (\frac{4  \lambda_1}{r - r_+})^2 n {K_+^a}^{n-1} e^{- \frac{{K_+^a}(r - r_+)}{4  \lambda_1}} \\
&~&~~~~~~~+ \cdots + (\frac{4  \lambda_1}{r - r_+})^{n+1} n! e^{- \frac{{K_+^a}(r - r_+)}{4  \lambda_1}}\\
&\leq& c \frac{4  \lambda_1}{r - r_+} n! {K_+^a}^n e^{- \frac{{K_+^a}(r-r_+)}{4 \lambda_1}},
\end{eqnarray*}
and with the definition of $K_+^a$ and\\ $$\log \frac{4  \lambda_1}{ r - r_+}+ n \log K_+^a + \log n!-\frac{K_+ (r-r_+)}{4  \lambda_1} \leq  \log {\mu_\nu},$$\\ we finish the proof of assumption (\ref{5a1}).

Due to
\begin{eqnarray*}
2 \mu_\nu^{2\sigma} (M^{*a} + 1) ([\log \frac{1}{\mu_\nu}] +1 )^{3 \eta (\tau +1)} < r-r_+,
\end{eqnarray*}
we finish the proof of assumption (\ref{5c1}). Assumption (\ref{5b1}) is obvious by (\ref{580a}) and (\ref{531}), and we omit the detail.

 Above all, KAM steps described above are valid for all $\nu$, which gives the desired sequences stated in the lemma.
 Now, we accomplish the proofs of $\bf{(1)}$, $\bf{(2)}$ and $\bf{(3)}$.

 The proof of $\bf{(4)}$ is standard. The details can be found in \cite{LLL}.

\end{proof}

\subsection{Convergence and Measure Estimate}\label{5107}
Let $\Psi_{\nu }^a = \Phi_1^a \circ \Phi_2 ^a\circ \cdots \circ \Phi_\nu^a, ~ \nu = 1,2, \cdots.$ Then $\Psi_{\nu}^a: \tilde{D}_{\nu} \times G _0 ^a \rightarrow \tilde{D}_0$, and
  \begin{eqnarray}
   \nonumber H^a \circ \Psi_{\nu}^a &=& H_{\nu}^a = N_{\nu}^a+ P_{\nu}^a(\frac{x}{\lambda}, y, \xi),\\
   \nonumber  N_{\nu}^a&=& e_{\nu}^a + \langle \omega_\nu^a , y\rangle+ h_\nu^a (y, \xi),~~\nu= 0,1,\cdots,
  \end{eqnarray}
  where $\Psi _0^a = id $.

  Simply, $N_\nu^a$ converges uniformly to $N_\infty^a$, $P_\nu^a$ converges uniformly to $P_\infty^a$
 and $\partial _y^j P_\infty^a = 0.$ For details, refer to \cite{LLL}.

Hence for each $\xi\in G_{*}^a$, $T^d \times \{0\}$ is an analytic invariant torus of $H_\infty^a$ with the toral frequency $\omega_\infty^a$, which for all $k\in Z^d \backslash \{0\},~1\leq q\leq n$, by the definition of $G_\nu^a$, satisfies the fact when $\bf{(H1)}$ holds,
\begin{eqnarray}
   \nonumber |\langle k, \omega_\infty^a \rangle| > \frac{\gamma}{2|k|^\tau} , ~~(\omega_\infty ^a)_q \equiv (\omega_0^a )_q.
\end{eqnarray}

Following the Whitney extension of $\Psi_\nu^a,$ all $e_\nu^a,$ $\omega_\nu^a,$ $h_\nu^a,$ $P_\nu^a,$ $\nu = 0,1,\cdots,$ admit uniform $C^{d-1 +\sigma_0}$ extensions in $\xi \in G_0^a$ with derivatives in $\xi$ up to order $d-1$ satisfying the same estimates (\ref{523}) - (\ref{528}). Thus, $e_\infty^a$, $\omega_\infty^a$, $h_\infty^a$, $P_\infty^a$ are $C^{d-1}$ Whitney smooth in $\xi \in G_{*}^a$, and the derivatives of $e_\infty^a -e_0^a$, $\omega_\infty^a -\omega_0^a$, $h_\infty^a -h_0^a$ satisfy similar estimates as (\ref{524}), (\ref{526}), (\ref{528}). Consequently, the perturbed tori form a $C^{d-1}$ Whitney smooth family on $G_{*}^a$.

The measure estimate is the same as one in \cite{LLL}. For the sake of simplicity, we omit the details. Thus the proof of Theorem \ref{dingli9} is complete.

\section{Slow Rotation Case}\setcounter{equation}{0} \label{lower}
In this section, we consider the Hamiltonian systems with slow rotation perturbation, i.e., Hamiltonian systems of the following form:
\begin{eqnarray}\label{502}
H^b(x,y)=N^b (y) + \varepsilon P^b(\lambda_2 x,y,\varepsilon),
\end{eqnarray}
defined on the complex neighborhood $D(r,s) = \{(x,y): |{\rm Im}~x| < r, |y|<s\}$ of $\mathbb T^d \times \{0 \}  \subset\mathbb T^d \times\mathbb R^d$, where $\lambda_2 = \varepsilon^{\beta}$, $\beta \in (0, \frac{ \sigma^2 }{ (d+ m+5)\tau +d+m+9})$, $\sigma,$ $m$ and $\tau$ are defined in Theorem \ref{dingli01}, $N^b (y)$ is a real analytic function, $\det {\partial_y^2 N^b} \neq 0$, on a complex neighborhood of the bounded closed region $G^b$, $\varepsilon P^b(\lambda_2 x, y)$, a small perturbation, is a real analytic function, where $\varepsilon > 0$ is a small parameter.

Assume
\begin{itemize}
\item[\bf{(S2)}] There exists an $N > 1$ such that
\begin{eqnarray*}
{\rm rank} \{ \partial_{y}^\alpha N^b: 1\leq |\alpha|\leq N,~~~~ \forall y \in G^b \} = d.
\end{eqnarray*}
\item[\bf{(H2)}] $A^b$ has an $n \times n$ nonsingular minor $\mathcal{A}^b$.
\end{itemize}
For Hamiltonian system (\ref{502}), we have the following.

\begin{theorem}\label{dingli01}
Let $H^b$ be analytic and $\beta \in (0, \frac{ \sigma^2 }{ (d+ m+5)\tau +d+m+9})$, where $\sigma \in (0, \frac{1}{3})$, $m >1$, $\tau > d(d - 1) -1$ are given. Under assumptions $\bf{(S2)}$ and $\bf{(H2)}$, there exist a $\varepsilon_0 >0 $ and a family of Cantor sets $G_\varepsilon^b \subset G^b$, $0<\varepsilon < \varepsilon_0$, such that for each $y \in G_\varepsilon^b$, the unperturbed $d-$tours $T_{y}^b$ persists and gives rise to a real analytic invariant $d-$torus $T_{\varepsilon, y}^b$ preserving $n$ corresponding unperturbed toral frequencies. Moreover, the relative Lebesgue measure $|G^b \setminus G_\varepsilon^b |$ tends to 0 as $\varepsilon \rightarrow 0$.
\end{theorem}

First we consider the transformation: $y\rightarrow y +\xi, ~~x \rightarrow x,$  where $ \xi \in G^b$. Then Hamiltonian system (\ref{502}) reads
\begin{eqnarray}
\label{585}H^b(x,y,\xi) &=& N^b(y,\xi) + \varepsilon P^b(\lambda_2 x,y,\xi),\\
\nonumber N^b(y,\xi) &=& e^b(\xi) + \langle \omega^b(\xi), y \rangle + h^b(y,\xi),
\end{eqnarray}
where $h^b = \frac{1}{2}\langle y, A^by\rangle + \hat{h}^b$, $\hat{h}^b  = O(|y|^3),$ $\lambda_2 = \varepsilon^{\beta}$, $\beta \in (0, \frac{ \sigma^2 }{ (d+ m+5)\tau +d+m+9})$.

Denote $P_0 ^b=  \varepsilon P^b(\lambda_2 x,y,\xi)$, then with the Cauchy estimate, obviously,
\begin{eqnarray}\label{507}
|\partial_\xi^l P_0^b|_{D(r,s)} \leq c \gamma_0^{d+m+5} s_0^m \mu_0, \quad |l| < d,
\end{eqnarray}
where $c$ is a constant.

Thus, we have
\begin{eqnarray*}
H^b (x,y,\xi)&=& N_0^b (y,\xi) +  P_0^b(  \lambda_2 x,y,\xi),\\
\nonumber N^b(y,\xi) &=& e_0^b + \langle \omega_0^b(\xi), y \rangle + h_0^b(y,\xi),
\end{eqnarray*}
with
\[h_0^b(y,\xi)= \frac{1}{2}\langle y, A_0^b(\xi) y\rangle + \hat{h}_0^b(y, \xi), \quad \hat{h}_0^b (y,\xi) = O(|y|^3),
\]
 where
 $\omega_0^b(\xi) = \partial_y N_0^b (\xi)$, $A_0^b(\xi) = \partial_y ^2 N_0^b(\xi)$, $\lambda_2= \varepsilon^{\beta}$, $\beta \in (0, \frac{ \sigma^2 }{ (d+ m+5)\tau +d+m+9 })$. Moreover,
\begin{eqnarray*}
|\partial_\xi^l P_0^b|_{D(r,s)} \leq c \gamma_0^{d+m+5} s_0^m \mu_0, ~~~~~~~~~|l| < d.
\end{eqnarray*}

\subsection{KAM step}\label{KAM step 3}

Suppose that after $\nu-$th step, we have arrived at the real analytic Hamiltonian system of the following form:
\begin{eqnarray}
\label{568}H^b (x,y,\lambda,\xi)&=& N^b (y,\xi) + \varepsilon P^b(\lambda_2 x,y,\xi),\\
\nonumber N^b(y,\xi) &=& e^b + \langle \omega^b(\xi), y \rangle + h^b(y,\xi),\\
\nonumber h^b(y,\xi) &=& \frac{1}{2} \langle y , A^b y\rangle + \hat{h}^b(y,\xi),\\
\label{589} |\partial_\xi^l P^b(\lambda_2 x, y,\xi)|_{D(r,s)} &\leq& c \gamma^{d+m+5} s^m \mu, ~~~~~~~~~|l| < d,
\end{eqnarray}
where $y\in G^b \subset\mathbb R^d,$ $x \in\mathbb T^d$, $\xi \in G^b$, $\lambda = \varepsilon^{\beta}$, $\beta \in (0, \frac{ \sigma^2 }{ (d+ m+5)\tau +d+m+9})$, $\hat{h}^b = O(|y|^3)$.

By considering both averaging and translation, we will find a symplectic transformation $\Phi _ {+}^b$, which, on a small phase domain $D(r_+, s_+)$  and a smaller parameter domain $G_+^b$, transforms Hamiltonian (\ref{568}) into the Hamiltonian of the next KAM step, i.e.
 \begin{center}
  $H_{+}^b = H^b{\circ} {\Phi _+^b}= {N_+^b} +{P_+^b}$,
 \end{center}
where $N_+^b$, $P_+^b$ enjoy similar properties as $N^b$, $P^b$, respectively.

\subsubsection{{Truncation}}
Consider the Taylor - Fourier series of $P^b(\lambda_2 x,y,\xi)$,
 \begin{eqnarray*}
   P^b(\lambda_2 x,y,\xi) = \sum _{|k| \in Z^d,~ |\jmath|\in Z_{+}^d} P_{k\jmath}^b {y^\jmath} {e^{\sqrt{-1} \lambda_2\langle k, x\rangle} },
\end{eqnarray*}
and let $R^b(\lambda_2 x,y,\xi)$ be the truncation of $P^b(\lambda_2 x,y,\xi)$ of the form:
\begin{eqnarray}
\nonumber R^b(\lambda_2 x,y,\xi) = \sum _{|k| \leq K_+^b,~ |\jmath| \leq m} P_{k\jmath}^b {y^\jmath} {e^{\sqrt{-1} \lambda_2\langle k, x\rangle} }.
\end{eqnarray}
Using same techniques as ones in \cite{LLL}, on $D_{{\frac{7}{8}} \alpha},$ we have
 \begin{eqnarray}
  \nonumber |\partial_{\xi}^{l}(P ^b - R^b)|_{D_{\frac{7}{8} {\alpha}}} &\leq& c {\gamma}^{d+ m+5} {s^m} {\mu }^2,\\
   \nonumber |\partial_{\xi}^{l} R^b |_{D_{{\frac{7}{8} {\alpha}}}} &\leq&  c {\gamma} ^{d+ m +5} {s^m} {\mu},
 \end{eqnarray}
provided
\begin{eqnarray}\label{5a2}
 {\int _{K_+^b}^{\infty}} {t^{d}} {e^{-t \frac{\lambda_2(r- r_+)}{4}}} dt \leq \mu.
\end{eqnarray}

\subsubsection{Homology Equation}
Denote the homology equation by
 \begin{eqnarray}\label{5Posion2}
  \{N^b,F^b\}+ R^b-[R^b]=0,
\end{eqnarray}
where \begin{eqnarray}\label{535}
F^b = \sum _{0 < |k| \leq K_+^b,~ |\jmath| \leq m} f_{k\jmath}^b {y^\jmath} {e^{\sqrt{-1} \lambda_2\langle k, x\rangle} },
 \end{eqnarray}
and $[R^b]= \int _{\mathbb T^d} R^b dx$ is the average of truncation $R^b(\lambda_2 x, y, \xi)$.

Contracting coefficients between two sides of (\ref{5Posion2}), we have
\begin{eqnarray*}
\sqrt{-1}\lambda_2 \langle k, \omega^b+ \partial_y h^b\rangle f_{k\jmath}^b = P_{k\jmath}^b.
\end{eqnarray*}
Denote $M^{*b} = \max \limits_{|l|\leq d, |j|< m+5, |y|\leq \beta_0} |\partial_\xi^l \partial_y^j h_0^b (y,\xi)|$. Under the assumptions
\begin{eqnarray}
\label{5b2} \max\limits_{|l|\leq d, |j|< m+5} |\partial_\xi^l \partial_y^j h^b - \partial_\xi^l \partial_y^j h_0^b|_{D(s)\times G_+^b} &\leq& \mu_0^{\frac{1}{2}},\\
\label{5c2} \frac{\gamma - \gamma_+}{(M^{*b} +1){ K_+^b}^{\tau + 1}}&>& 2s,
\end{eqnarray}
using the same method as subsection ${\bf{\ref{5135}}}$, we have
\begin{eqnarray}\label{538}
 |\partial_\xi^l \partial_y^j f_{k\jmath}^b| \leq \left\{
                                                      \begin{array}{ll}
                                                        c \mu |k|^{(|j|+ |l|+1)\tau +|j|+|l|+1}s^{m - |j|} \frac{1}{\lambda_2} e^{-\lambda_2 |k|r}, & \hbox{$|j| \leq m$;} \\
                                                        c \mu |k|^{(|j|+ |l|+1)\tau +|j|+|l|+1} \frac{1}{\lambda_2} e^{-\lambda_2 |k|r}, & \hbox{$m< |j|\leq m+4$,}
                                                      \end{array}
                                                    \right.
\end{eqnarray}
for all $(y,\xi) \in D(s) \times G_+^b$, $0 < |k|\leq K_+^b$, $|l| \leq d$.

Therefore, we have
\begin{eqnarray}
\nonumber |\partial_\xi^l \partial_x^i \partial_y^j F^b| &=& |\sum_{0<|k|\leq K_+^b, |\jmath|\leq m}(\sqrt{-1} k\lambda_2)^i \partial_y^j(\partial_\xi^l f_{k\jmath}^b y^\jmath e^{\sqrt{-1} \lambda_2\langle k,x\rangle}) |\\
\label{545} &\leq& \left\{
                     \begin{array}{ll}
                       c s^{m - |j|} \mu \Gamma^b(r - r_+), & \hbox{$|j|\leq m$;} \\
                       c \mu \Gamma^b(r - r_+), & \hbox{$m<|j|\leq m+4,$}
                     \end{array}
                   \right.
\end{eqnarray}
where $\Gamma^b(r - r_+) = \sum\limits_{\substack{0<|k|\leq K_+^b,\\ |i|, |j|\leq m+4}} |k|^{(|l|+ |j|+1) \tau + |l|+ |i|+|j|+1} {\lambda_2^{|i| -2}} e^{-\frac{  |k|\lambda_2(r -r_+)}{8}}$.

\subsubsection{Frequency Retention}
Let $\mathbf{y}^b$ and $p_{01}^b$ be the vectors formed by the $n$ components of $y$ and $P_{01}^b$, respectively, and denote $ \hat{h}^b({y})= \hat{h}^b((\mathbf{y}^b,0)^T).$ Then by the implicit function theorem, the equation
\begin{eqnarray}
\label{540} \mathcal{A}^b\mathbf{y}^b + {\partial_{\mathbf{y}^b}}\hat{h}^b(\mathbf{y}^b) &=& - p_{01}^b
\end{eqnarray}
admits a unique solution $\mathbf{y}_{*}^b$ on $D(s)$, which also smoothly depends on $\xi$, where $\mathcal{A}^b$ is $n\times n$ nonsingular minor of $A^b$. Define $ y_{*}^b= (\mathbf{y}_{*}^b,0 )^T$. By (\ref{540}), we clearly have
 \begin{eqnarray*}
     A^b y_{*}^b + \partial _{y} {\hat{h}^b(y_{*}^b)} = -(p_{01}^b,0 )^T.
   \end{eqnarray*}

Then under the time $1-$map $\Phi _{F^b}^1$ of the flow generated by a Hamiltonian $F^b$ and the transformation $\phi^b : x \rightarrow x, ~y \rightarrow y+ y_{*}^b,$  we have
 \begin{eqnarray}\label{541}
    H_+^b = H^b \circ \Phi_{F^b}^1 \circ \phi  = e_{+}^b + \langle \omega_+^b , y\rangle + h_{+}^b(y)+P_+^b(\lambda_2 x,y,\xi),
 \end{eqnarray}
with
\begin{eqnarray}
 \nonumber e_+^b &=& {e}^b + \langle \omega^b, y_{*}^b\rangle +\frac{1}{2} \langle y_{*}^b ,A^b y_{*}^b\rangle + \hat{h}^b(y_{*}^b)+ [R^b](y_*^b),~~ \omega_+^b = \omega^b +  P_{01}^b -\left(
                                                   \begin{array}{c}
                                                      p_{01}^b \\
                                                     0 \\
                                                   \end{array}
                                                 \right),\\
\nonumber A_+^b &=& A^b + \partial_y^2 \hat{h}^b(y_*^b) + \partial_y^2 [R^b](y_*^b),~~ h_+^b (y) = \frac{1}{2}\langle y, A_+^b y\rangle + \hat{h}_+^b (y),~~\hat{h}_+^b = O(|y|^3),\\
\nonumber \hat{h}_+^b&=& \hat{h}^b (y+ y_{*}^b)- \hat{h}^b(y_{*}^b)- \langle\partial_{y}\hat{h}^b(y_{*}^b) , y \rangle - \frac{1}{2} \langle y, \partial_{y}^2\hat{h}^b(y_{*}^b)y \rangle + [R^b](y+y_*^b) \\
\nonumber &~&~~~ - [R^b](y_*^b)- \langle \partial_y[R^b](y_*^b), y\rangle- \frac{1}{2}\langle y, \partial_y^2 [R^b](y_*^b) y\rangle,\\
\nonumber P_+^b &=& (\int_0^1 \{R_t^b, F^b\} \circ \Phi_{F^b}^t dt+ (P^b- R^b)\circ \Phi_{F^b}^1) \circ \phi^b + \psi^b,\\
\nonumber \psi^b &=& \langle \partial_y [R^b](y_*^b), y\rangle  - \langle P_{01}^b, y\rangle.
\end{eqnarray}

\subsubsection{ Estimate on $N_+^b$}
Denote $M_{*}^b = \max\limits_{\xi \in G_0^b} |{\mathcal{A}_0^b}^{-1}(\xi)| +1$ and let $\mu_0$ small enough, say, $\mu_0 < \frac{1}{8{(M_*^b)}^2(M^{*b} +1)}$, such that $M_*^b(M^{*b} +1)s_0^2 < \frac{1}{4}$.
For $\xi \in G_+^b$, we denote
\begin{eqnarray*}
B^b(y,\xi) = \mathcal{A}^b + \int _0^1 \partial_y^2 \hat{h}^b(\theta y) d\theta.
\end{eqnarray*}
 Then, by (\ref{540}),
 \begin{eqnarray}\label{542}
  B^b(\mathbf{y}_{*}^b)\mathbf{y}_{*}^b = - p_{01}^b.
\end{eqnarray}
Under the following assumption
\begin{eqnarray}\label{5d2}
    4  M_{*}^b (M^{*b}+1)\gamma ^{d+ m+5}s^{m-1} \mu < \frac{1}{2},
\end{eqnarray}
with the similar techniques to subsection \ref{Estimation N_+}, we have
 \begin{eqnarray}
\nonumber |\partial _{\xi}^l y_{*}^b|&<& c M_{*}^b \gamma^{d + m + 5} s^{m - 1} \mu,\\
 \nonumber |\partial_{\xi}^{l} e_+^b - \partial_{\xi}^{l}e^b| _{G_+^b} &\leq& c  \gamma ^{d+m+5} s^{m-1} \mu,\\
  \nonumber |\partial_{\xi}^{l}{\omega_+^b} - \partial_{\xi}^{l}{\omega^b}| _{G_+^b} &\leq& c \gamma ^{d+m+5} s^{m-1} \mu,\\
  \nonumber  |\partial_y ^j \partial_{\xi}^{l}h_+^b  - \partial_y ^j \partial_{\xi}^{l}{h^b}| _{G_+^b} &\leq& \left\{
                                                                                                                      \begin{array}{ll}
                                                                                                                        c \gamma ^{d+m+5}s^{m-|j|} \mu, & \hbox{$|j| \leq m$;} \\
                                                                                                                         c \gamma ^{d+m+5}\mu, & \hbox{$m< |j|\leq m+4$.}
                                                                                                                      \end{array}
                                                                                                                    \right.
 \end{eqnarray}

\subsubsection{Estimate on $\Phi_+^b$}
By (\ref{545}) and the Cauchy estimate on ${D_{\frac{3}{4}\alpha}}$,
\begin{eqnarray}
\label{546}  (r - r_+) |\partial _{\xi}^l \partial _y F^b|,~s| \partial _{\xi}^l \partial _x F^b|  ~\leq ~ c s^{m} \mu \Gamma^b(r -r_+).
\end{eqnarray}
Furthermore, we have, inductively, $|D^n \partial _\xi ^l F^b| \leq c \mu \Gamma^b(r -r_+),$ $n \leq 4.$

Denote $\Phi _{F^b}^t =(\phi_1^b,\phi_2^b)^T$, where $\phi_1^b$, $\phi_2^b$ are components of $\Phi_{F^b}^t$ in the directions of $x,$ $y,$ respectively. Let $(x, y)$ be any point in $D_{\frac{1}{4} \alpha}$ and let $t_{*} = sup \{ t\in [0,1] : \Phi_{F^b}^t(x,y )\in D_\alpha \}$. We note that $D_\alpha \subset \hat{D}(s).$ With assumptions
\begin{eqnarray}
\label{5f2} c \Gamma^b(r-r_+) \mu &<& \frac{1}{8} (r - r_+),\\
\label{5g2} c \Gamma^b (r - r_+) &<& \frac{1}{8} \alpha,
\end{eqnarray}
 we have
\begin{eqnarray*}
  |\phi_1^b(x, y) - x| &\leq& \int_0^t |F_y^b \circ \Phi_{F^b}^{u}|_{D_\alpha} du \leq|F_y^b| _{\hat{D}(s)}< \frac{1}{8} (r - r_+),\\
  |\phi_2^b (x, y) - y| &\leq& \int_0^t |F_x^b \circ \Phi_{F^b}^{u}|_{D_\alpha} du  \leq|F_x^b| _{\hat{D}(s)}< \frac{1}{8} \alpha.
\end{eqnarray*}

Then $\Phi _{F^b}^t : D_{\frac{1}{4} \alpha} \rightarrow D_{\frac{1}{2} \alpha} \subset D_\alpha$.  Hence $t_* = 1$ and $\Phi_{F^b}^t: D_{\frac{\alpha}{4}} \rightarrow D_{\frac{\alpha}{2}}$. Under the assumption
\begin{eqnarray}\label{5e2}
c s^{m-1} \mu < \frac{1}{8} \alpha s,
\end{eqnarray}
 it is easy to check $\phi^b: D_{\frac{1}{8} \alpha} \rightarrow D_{\frac{1}{4}\alpha}.$

With the standard Whitney extension theorem (see {\cite{Poschel,Stein}}), it is easy to see that $F^b$ and $y_{*}^b$ can be extended to functions of H$\ddot{o}$lder class $C^{m+3,d-1+\sigma_0}(\hat{D}(\beta_0)\times {G_0^b})$, respectively, where $0< \sigma_0 <1$ is fixed. Moreover,
\begin{eqnarray}
\nonumber \|F^b\|_{{C^{m+3, d-1+\sigma_0}}(\hat{D}(\beta_0) \times G_0^b)} &\leq& c \mu \Gamma^b(r - r_+),\\
\nonumber \|y_{*}^b\| _{C^{d-1 + \sigma_0 }(G_0^b)}&\leq& c \mu \Gamma^b (r- r_+).
\end{eqnarray}
By Gronwall's inequality and assumptions
\begin{eqnarray}
\label{h1} c \mu \Gamma^b (r- r_+) &<& \frac{1}{8} (r - r_+),\\
\label{i1} c \mu \Gamma^b (r- r_+) &<& \beta - \beta_+,
\end{eqnarray}
inductively, on $\tilde{D}_+ \times G_0^b$, we have
\begin{eqnarray}\label{547}
  |\Phi_{F^b}^t - id |, |\partial _y \Phi _{F^b}^t - I_{2d}|, |\partial_y ^j \Phi _{F^b}^t| \leq c \mu \Gamma^b (r-r_+).
\end{eqnarray}
Therefore $\Phi _+^b = \Phi_{F^b}^1 \circ \phi^b : \hat{D}_+ \rightarrow D(r, \beta)$ is of classes $C^{m+2}$ and also depends $C ^{d-1+ \sigma_0}$ smoothly on $\xi \in G_0^b$, where $\sigma_0 $ is described as above. Moreover, $\| \Phi_+^b - Id\|_{C^{m+2,d-1+\sigma_0}(\tilde{D}_+ \times G_0^b)}\leq c\mu \Gamma^b(r- r_+).$

Hence, the new Hamiltonian reads
\begin{center}
   $ H^b\circ \Phi _+^b = N_+^b +P_+^b$,
\end{center}
with
\begin{eqnarray*}
    N_+^b &=& e_+^b + \langle \omega_+^b, y\rangle + h_+^b(y),\\
 \omega _+^b &=& \omega^b + P_{01}^b + {A}^b y_{*}^b +\partial_y \hat{h}^b(y_{*}^b),\\
    P_+^b &=& \bar{P}_+^b \circ \phi^b + \psi^b ,
\end{eqnarray*}
where $h_+^b (y)$, ${A}^b$, $\psi^b$, $\hat{h}^b (y)$ have the same forms as ones mentioned above, and $ |\langle k , \omega_+^b  \rangle | > \frac{\gamma_+ }{ |k|^{\tau} }$ for all $0< |k| \leq K_+^b$, $ \xi \in G _+^b$, which is obvious with the assumptions mentioned above.

\subsubsection{Estimate on $P_+^b$}
 Now
\begin{eqnarray*}
P_+^b = \bar{P}_+^b \circ \phi^b + \psi^b = (\int_0^1 \{R_t^b , F^b\}\circ \Phi _{F^b}^t dt + (P^b-R^b)\circ \Phi _{F^b}^1 )\circ \phi^b +\psi^b.
\end{eqnarray*}
By above estimates in this section, we see that, for all $|l| \leq d$, $0\leq t \leq 1$,
\begin{eqnarray}
   \nonumber  |\partial _\xi^ l \{R_t^b , F^b\}\circ \Phi_{F^b}^t|_{{D_{\frac{1}{4}\alpha}}\times G_+^b }&\leq& c  \gamma^{d+m+5} s^{m} {\mu}^2 \Gamma^b (r- r_+),\\
   \nonumber |\partial _\xi^l(P^b-R^b) \circ \Phi_{F^b}^1|_{{D_{\frac{1}{4}\alpha}}\times G_+^b }&\leq& c \gamma^{d+m+5} s^m {\mu}^2,\\
\nonumber |\partial_\xi^l \phi^b |_{D_+ \times G_+^b} &\leq& c \gamma^{d +m +5} s^{m-1} \mu,\\
   \nonumber|\partial_\xi^l \psi^b|~~~~~&\leq& c \gamma^{d+m+5} s^{m} {\mu}^2.
\end{eqnarray}
Hence, $|\partial _\xi^l P_+^b(\lambda_2 x, y,\xi) |_{D_+ \times G _+^b} \leq c \gamma^{d+m+5} s^m {\mu}^2 (\Gamma^b(r- r_+) +2 )$ for $|l|\leq d$. Let $c_0$ be the maximal one of the $c's$ we mentioned in this section  and define $\mu_+ = 8^m c_0 \mu^{1 + \sigma}.$ Under the assumption
\begin{eqnarray}\label{5j2}
 \mu^{\sigma } (\Gamma^b(r- r_+) +2) \leq \frac{\gamma_{+}^{d+m+5}}{\gamma^{d+m+5}},
 ~~~on~ D_{+} \times G_+^b,
\end{eqnarray}
we have
\begin{eqnarray}
\nonumber |\partial _{\xi}^l P_+^b| &\leq& 8^m c_0 s^m {\mu}^{1+ \sigma} {\mu}^{1-2\sigma} {\mu}^{\sigma} {\gamma}^{d+m+5} (\Gamma^b(r- r_+) +2)\\
 \nonumber &\leq& c_0\gamma_{+}^{d+m+5} s_{+}^m \mu _+,~~~~~~~~|l|\leq d.
\end{eqnarray}

We now complete one KAM step.

\subsection{Iteration Lemma}
Let $r_0$, $s_0$, $\gamma _0$, $\beta_0$, $\mu_0$, $H_0^b$, $N_0^b$, $e_0^b$, $\omega_0^b$, $h_0^b$, $A_0^b$, $\hat{h}_0^b$, $P_0^b$ be given as above.  We have the iteration lemma for (\ref{568}) as follows.
\begin{lemma}\label{lemma21}
If (\ref{589}) holds for a sufficiently small $\mu= \mu (r, s, d, \tau)$, then the KAM step described in subsection ${\bf{\ref{KAM step 3}}}$ is valid for all $\nu =0,1,\cdots,$ and sequences
$G_\nu^b,$ $H_\nu^b,$ $N_\nu^b,$ $e_\nu^b,$ $\omega_\nu^b,$ $h_\nu^b,$ $A_\nu^b,$ $\hat{h}_\nu^b,$ $P_\nu^b$, $\Phi_\nu^b$, $\nu = 1,2,\cdots ,$ possess the following properties:
\begin{itemize}
\item[{\bf (1)}]$\Phi _\nu^b: \hat{D}\times G_0^b \rightarrow \hat{ D}_{\nu-1}$, $D_\nu \times G_\nu^b \rightarrow D_{\nu-1}$ is symplectic for each $\xi\in G_0^b$ or $G_\nu^b$, and is of class $C^{m+2,d-1+\sigma_0}$, $C^{\iota, d}$, respectively, where $\iota$ stands for real analyticity, $0<\sigma_0<1$ is fixed, and
 \begin{eqnarray}\label{548}
 \|\Phi_\nu^b - id\| _{C^{m+2,d-1+\sigma_0}(\hat{D}_\nu \times G^b)} \leq  \frac{\mu^{\frac{1}{2}}}{2^{\nu}}.
 \end{eqnarray}
Moreover, on $\hat{D}_\nu \times G^b$,
\begin{center}
  $H_\nu^b = H_{\nu-1}^b\circ \Phi_{\nu}^b= N_{\nu}^b+ P_\nu^b,$
  \end{center}
  where $N_\nu^b = e_\nu^b + \langle \omega_\nu^b, y\rangle + \frac{1}{2} \langle y, A_\nu^b y\rangle + \hat{h}^b(y),$ $A_\nu^b$ has an $n \times n$ minor $\mathcal{A}_{\nu}^b$, which is nonsingular on $G_\nu^b$, $ \hat{h}^b(y)= O (|y|^3)$;
\item[{\bf (2)}]Under assumption $\bf{(H2)}$, we have $(\omega_\nu^b)_q=(\omega_{\nu-1}^b)_q $, $\forall$ $\xi \in G_\nu^b,$   $q= 1,2,\cdots, n;$
\item[{\bf(3)}]For all $|l|\leq d ,$
  \begin{eqnarray}
\label{549} |\partial_\xi^l e_\nu^b -\partial_\xi^l e_{\nu -1}^b | _{G_\nu^b}&\leq &\gamma _0^{d+m+4} \frac{\mu}{2^{\nu}};\\
\label{550}|\partial_\xi^l e_\nu^b -\partial_\xi^l e_0^b | _{G_\nu^b} &\leq& \gamma _0^{d+m+4} \mu;\\
\label{551}|\partial_\xi^l {\omega}_{\nu}^b -\partial_\xi^l {\omega}_{\nu -1}^b |_{G_\nu^b} &\leq& \gamma _0^{d+m+4} \frac{\mu}{2^{\nu}};\\
\label{552} |\partial_\xi^l {\omega}_\nu^b -\partial_\xi^l {\omega}_0^b |_{G_\nu^b}&\leq& \gamma _0^{d+m+4} \mu;\\
\label{553} |\partial _\xi^l h_\nu^b  -\partial _\xi^l  h_{\nu-1}^b |_{D_\nu \times G_\nu^b} &\leq& \gamma _0 ^{d+m+4} \frac{{\mu}^{\frac{1}{2}}}{2^{\nu}};\\
\label{554} |\partial _\xi^l h_\nu^b  -\partial _\xi^l h_0^b |_{D_\nu \times G_\nu^b} &\leq& \gamma _0 ^{d+m+4} {\mu}^{\frac{1}{2}};\\
\label{555} |\partial_{\xi}^l P_\nu^b|_{D_\nu \times G_\nu^b}&\leq& \gamma_ \nu ^{d+m+5} s_\nu^m {\mu_\nu};
 \end{eqnarray}
 \item[{\bf(4)}] \begin{eqnarray*}
G_\nu^b &=& \left\{ \xi \in G _{\nu -1}^b : |\langle k, \omega _{\nu-1}^b(\xi)\rangle|> \frac{\gamma_{\nu-1}}{|k|^\tau},  ~for ~all~0< |k| \leq K_\nu^b \right\}.
 \end{eqnarray*}
\end{itemize}
\end{lemma}

\begin{proof}
Actually, it suffices to verify the assumptions that we put forward in Section \ref{lower} for all $\nu$. For simplicity, we let $r_0 = \beta_0= 1$. By choosing $\mu_0$ small, we also see that other assumptions are hold for $\nu =0 $. By the definition of $\mu_\nu$,
\begin{eqnarray}\label{557}
    \mu_{\nu} = 8^m c_0 \mu_{\nu -1}^{1+ \sigma} < \cdots < \frac{1}{{\zeta}^{\nu}} \mu_0,
\end{eqnarray}
where $\zeta \gg1$ and \begin{eqnarray}\label{558}
  \mu _0 < (\frac{1}{8^m c_0 \zeta})^{\sigma }\ll 1.
\end{eqnarray}
Then assumption (\ref{5d2}) holds. Besides, using (\ref{561}), (\ref{5e2}) is obvious. Note
\begin{eqnarray*}
&~&\int_{K_+^b}^\infty t^n e^{- \frac{t \lambda_2 (r -r_+)}{4}} d t \\
&\leq& \frac{4}{\lambda_2 (r -r_+)} {K_+^b}^n e^{-\frac{K_+^b \lambda_2 (r- r_+)}{4}} + \frac{4^2}{\lambda_2^2 (r -r_+)^2} n {K_+^b}^{n-1} e^{-\frac{K_+^b \lambda_2 (r- r_+)}{4}}+ \cdots\\
&~&+\frac{4^{n+1}}{\lambda_2^{n+1} (r -r_+)^{n+1}} n! e^{-\frac{K_+^b \lambda_2 (r- r_+)}{4}}\\
&\leq& \frac{4^{n+1}}{\lambda_2^{n+1} (r -r_+)^{n+1}} {K_+^b}^n n! e^{-\frac{K_+^b \lambda_2 (r- r_+)}{4}}
\end{eqnarray*}
and
\begin{eqnarray*}
\log \frac{4^{n+1}}{\lambda_2^{n+1} (r -r_+)^{n+1}} + n \log {K_+^b} + \log n!  - \frac{K_+^b \lambda_2(r -r_+)}{4} <\log \mu,
\end{eqnarray*}
(\ref{5a2}) is obvious. The proof of (\ref{5c2}) is equivalent to one of
\begin{eqnarray}\label{559}
2 s (M^* + 1 ) {K_+^b}^{\tau + 1} < r -r_+.
\end{eqnarray}
With the definition of $s_\nu$ and $K_+^b$, for $\beta \in (0, \frac{1}{3 m(\tau+1)})$ we have
\begin{eqnarray*}
2 (\frac{1}{8})^\nu (8^m C_0)^{(\frac{(1 + \sigma)^\nu - (1+\sigma)}{\sigma^2} - \frac{\nu}{\sigma}) \frac{1}{m+1}} {\mu_0}^{\frac{(1 + \sigma)^\nu - 1}{\sigma} \frac{1}{m + 1}} s_0 (M^{*b} + 1) ([\frac{1}{\lambda_2}] +1)^{2(\tau +1)}\\
 ([\log \frac{1}{(8^m C_0)^{\frac{(1 + \sigma)^\nu -1}{\sigma}}}] +[\log(\frac{1}{\mu_0})^{(1 + \sigma)^\nu}]  +3 )^{3 \eta (\tau +1)} < r -r_+,
\end{eqnarray*}
 and thus (\ref{559}) holds.
Since
\begin{eqnarray}
 \nonumber \Gamma _\nu^b &\leq& \int _1^{\infty} {t}^{(|l| +|j|+ 1)\tau + |l| +|i| + |j| +2}\lambda_2^{|i| - 2} e^{-\frac{t\lambda_2(r - r_+)}{8}} dt\\
 \nonumber  &\leq& \lambda_2^{|i| - 2} ( \frac{ 2^{\nu+6}}{\lambda_2}  e^{-\frac{\lambda_2}{2^{\nu+6}}} +  \frac{2^{2(\nu+6)}}{\lambda_2^2} ((|l| + |j|+ 1)\tau + |j|+ |l|+ |i|+2) e^{-\frac{\lambda_2}{2^{\nu+6}}}\\
\nonumber &~&~~+ \cdots + \frac{2^{(\nu+6)((|l| + |j|+ 1)\tau + |j|+ |l|+ |i|+2)}}{\lambda_2^{(|l| + |j|+ 1)\tau + |j|+ |l|+ |i|+2}}\\
\nonumber &~&~~~~((|l| + |j|+ 1)\tau + |j|+ |l|+ |i|+2)! e^{-\frac{\lambda_2}{2^{\nu+6}}} )\\
\nonumber &\leq&~~ c \frac{2^{(\nu+6)((|l| + |j|+ 1)\tau + |j|+ |l|+ |i|+2)}}{\lambda_2^{(|l| + |j|+ 1)\tau + |j|+ |l|+4}}\\
 \nonumber&~&~~~~~((|l| + |j|+ 1)\tau + |j|+ |l| + |i|+2)! e^{-\frac{\lambda_2}{2^{\nu+6}}},
\end{eqnarray}
where $\Gamma_{\nu}^b = \Gamma^b(r_\nu - r_{\nu-1})$,
it is clear that
\begin{eqnarray*}
\mu_{\nu}^{\sigma} {\Gamma_\nu^b} &<& (8^m c_0)^{{(1+ \sigma)}^{\nu} -1} {\mu_0}^{\sigma (1+ \sigma)^{\nu}} \frac{2^{(\nu+6)((|l| + |j|+ 1)\tau + |j|+ |l| + |i|+2)}}{\lambda_2^{(|l| + |j|+ 1)\tau + |j|+ |l|+4}}\\
&~&\cdot ((|l| + |j|+ 1)\tau + |j|+ |l|+|i|+2)! e^{-\frac{\lambda_2}{2^{\nu+6}}}.
\end{eqnarray*}
When $\beta \in (0, \frac{ \sigma^2 }{ (|l|+ |j|+1)\tau +|l|+|j|+5})$, with sufficiently small $\varepsilon$ we have
\begin{eqnarray*}
\mu_{\nu}^{\sigma} {\Gamma_\nu^b} &=& \varepsilon^{( (1+\sigma)^\nu - \varrho) \sigma^2 }  (8^m c_0)^{{(1+ \sigma)}^{\nu} -1} 2^{(\nu+6)((|l| + |j|+ 1)\tau + |j|+ |l|+ |i|+2)}\\
&~& ((|l| + |j|+ 1)\tau + |j|+ |l|+ |i|+2)! e^{-\frac{\lambda}{2^{\nu+6}}}\\
&\leq&\frac{\gamma_{\nu+1}^{d +m +5}}{\gamma_{\nu}^{d +m +5}},
\end{eqnarray*}
where $\varrho = \frac{(|l|+ |j|+1)\tau +|l|+|j|+4}{(|l|+ |j|+1)\tau +|l|+|j|+5}$. Therefore, (\ref{5j2}) and (\ref{5f2}) hold for all $\nu \geq 1$. Besides, \begin{eqnarray*}
(1+ \sigma)^\nu (3 m^2+ (3 m^2 + 3m)(\sigma - \sigma^2) - 3m(1+\sigma) ) - m^2  + 3m -2>0,
\end{eqnarray*}
where $\sigma = \frac{1}{2(m+1)}$, (\ref{5g2}) holds.  Assumption (\ref{5b2}) is obvious, and we omit the detail.

 Above all, the KAM steps described in Section \ref{lower} are valid for all $\nu$, which gives the desired sequences stated in Lemma \ref{lemma21}.
 Now, we accomplish the proofs of $\bf{(1)}$, $\bf{(2)}$ and $\bf{(3)}$.

 The proof of $\bf{(4)}$ is standard. The details can be found in \cite{LLL}.
\end{proof}

The convergences and measure estimates are similar to subsection \ref{5107}. And we omit the detail. Thus the proof of Theorem \ref{dingli01} is complete.

\section{Multiscale Rotation Case}\setcounter{equation}{0} \label{couple}
In this section we show the persistence of invariant tori for a Hamiltonian system with multiscale rotation perturbation, i.e., (\ref{5131}), which is equal to (\ref{5132}).

In this section, we define $\partial_\xi^l = \partial_{\xi_1}^{l_1} \partial_{\xi_2}^{l_2} \partial_{\xi_3}^{l_3}, |l| =|l_1| + |l_2|+ |l_3|,$ $|i| = |i_1|+ |i_2|+|i_3|$, $|\mathbf{k}| = |k_1|+|k_2|+|k_3|,$ $|\imath| = |\imath_1|+|\imath_2|+|\imath_3|,$ $|\kappa| = |\kappa_1|+|\kappa_2|+|\kappa_3|,$ and denote $P_0^d=  \varepsilon P(\frac{x}{\lambda_1},y,\theta,\eta,\lambda_2 \varphi,I,\xi)$. By the Cauchy estimate,
\begin{eqnarray*}
|\partial_\xi^l P_0^d|_{D^d(r,s)} \leq c \gamma_0^{d+m+5} s_0^{m} \mu_0, ~~~~~~~~~|l| < d,
\end{eqnarray*}
where $c$ is a constant.

In other words, we have
\begin{eqnarray}
\label{5113} H ^d(\frac{x}{\lambda_1},y,\theta,\eta,\lambda_2 \varphi,I,\xi)&=& N_0^d(y,\eta, I, \xi) +P_0^d(\frac{x}{\lambda_1},y,\theta,\eta,\lambda_2 \varphi,I,\xi),~~\\
\nonumber N_0^d(y,\eta,I,\xi) &=& e_0^d + \langle \mathbf{a}_0, \mathbf{b}\rangle + h_0^d(y,\eta,I,\xi),\\
 \nonumber  h_0^d(y,\eta,I,\xi) &=& \frac{1}{2}\langle {\mathfrak{A}_0^d} \mathbf{b}, \mathbf{b}\rangle + \hat{h}_0^d(y, \eta, I,\xi),
\end{eqnarray}
where $\hat{h}_0^d (y, \eta, I,\xi)$ are all terms with the form of $y^{\iota_1} \eta^{\iota_2} I^{\iota_3}$, $|\iota_1|+ |\iota_2|+|\iota_3|\geq 3$, in $N_0^d$,  $\lambda_1 = \varepsilon^{\alpha}$, $\lambda_2 = \varepsilon^\beta,$ $\alpha \in\mathbb R_+^1$, and $\beta \in (0, \frac{\sigma^2}{3[(d+m+5)\tau + d+ 2m +13]})$. Moreover,
\begin{eqnarray}\label{5114}
|\partial_\xi^l P_0^d |_{D^d(r,s)} \leq c \gamma_0^{d+m+ 5} s_0^{ m} \mu_0, ~~~~~~~~~|l| < d.
\end{eqnarray}

\subsection{KAM step}\label{KAM step3}
Suppose that after $\nu-$th step, we have arrived at the real analytic Hamiltonian system of the following form:
\begin{eqnarray}
\label{5115} H^d (\frac{x}{\lambda_1},y,\theta,\eta,\lambda_2 \varphi,I,\xi)&=& N_\nu^d(y,\eta, I, \xi) +P_\nu^d(\frac{x}{\lambda_1},y,\theta,\eta,\lambda_2 \varphi,I,\xi),\\
\nonumber N_\nu^d(y,\eta, I, \xi) &=& e_\nu^d + \langle \mathbf{a}_\nu, \mathbf{b}\rangle + h_\nu^d(y,\eta,I,\xi),\\
 \nonumber  h_\nu^d (y,\eta,I,\xi) &=& \frac{1}{2}\langle \mathbf{b}, \mathfrak{A}_\nu^d \mathbf{b}\rangle+ \hat{h}_\nu^d(y, \eta, I, \xi),
\end{eqnarray}
where $\hat{h}_\nu^d (y, \eta, I)$ are all terms with the form of $y^{\iota_1} \eta^{\iota_2} I^{\iota_3}$, $|\iota_1|+ |\iota_2|+|\iota_3|\geq 3$, in $N_\nu^d$, $\lambda_1 = \varepsilon^{\alpha}$, $\lambda_2 = \varepsilon^\beta,$ $\alpha \in\mathbb R_+^1$, and $\beta \in (0, \frac{\sigma^2}{3[(d+m+5)\tau + d+ 2m +13]})$. Moreover,
\begin{eqnarray}\label{5116}
|\partial_\xi^l P_\nu^d|_{D^d(r,s)} \leq c \gamma_0^{d+m+5} s_0^{ m} \mu_0, ~~~~~~~~~|l| < d.
\end{eqnarray}

By considering both averaging and translation, we need to find a symplectic transformation $\Phi _ {+}^d$, which, on a small phase domain $D^d(r_+, s_+)$  and a smaller parameter domain $G_+^d$, transforms Hamiltonian (\ref{5115}) into the Hamiltonian of the next KAM step, i.e.
 \begin{center}
  $H_{+}^d = H^d{\circ} {\Phi _+^d}= {N_+^d} +{P_+^d}$,
 \end{center}
where $N_+^d$, $P_+^d$ enjoy similar properties as $N^d$, $P^d$, respectively.

\subsubsection{{Truncation}}\label{5136}
Consider the Taylor - Fourier series of $P^d (\frac{x}{\lambda_1},y,\theta,\eta,\lambda_2 \varphi,I,\xi)$,
 \begin{eqnarray*}
    P^d= \sum _{\substack{|k_1|, |k_2|, |k_3| \in Z^d,\\ |\imath_1|, |\imath_2|, |\imath_3|\in Z_{+}^d}} P_{k_1k_2k_3\imath_1\imath_2\imath_3}^d {y^{\imath_1} \eta^{\imath_2} I^{\imath_3}} {e^{\sqrt{-1} \frac{\langle k_1, x\rangle}{\lambda_1}} e^{\sqrt{-1} \langle k_2, \theta \rangle} e^{\sqrt{-1}\lambda_2 \langle k_3, \varphi\rangle} },
\end{eqnarray*}
and let $R^d (\frac{x}{\lambda_1},y,\theta,\eta,\lambda_2 \varphi,I,\xi)$ be the truncation of $P^d (\frac{x}{\lambda_1},y,\theta,\eta,\lambda_2 \varphi,I,\xi)$ of the form:
\begin{equation}
\nonumber     R^d= \sum _{|\mathbf{k}| \leq K_+^b, |\imath|\leq m} P_{k_1k_2k_3\imath_1\imath_2\imath_3}^d {y^{\imath_1} \eta^{\imath_2} I^{\imath_3}} {e^{\sqrt{-1} \frac{\langle k_1, x\rangle}{\lambda_1}} e^{\sqrt{-1} \langle k_2, \theta \rangle} e^{\sqrt{-1}\lambda_2 \langle k_3, \varphi\rangle} }.
\end{equation}

With the help of the Cauchy estimate and the following assumptions
\begin{eqnarray}
\label{5d4'''} {\int _{\frac{{K_+^b}}{3}}^{\infty}} {t^{d}} {e^{- \frac{ t (r- r_+)}{4\lambda_1}}} dt \leq \mu^{\frac{1}{3}},\\
\label{5d4'}{\int _{\frac{{K_+^b}}{3}}^{\infty}} {t^{d}} {e^{- \frac{ \lambda_2 t(r- r_+)}{4}}}dt \leq \mu^{\frac{1}{3}},\\
\label{5d4''}{\int_{\frac{{K_+^b}}{3}}^{\infty}} {t^{d}} {e^{-\frac{t(r- r_+)}{4}}} dt \leq \mu^{\frac{1}{3}}
\end{eqnarray}
 \noindent on $D_{{\frac{7}{8}} \alpha}^d,$ we have
 \begin{eqnarray}
  \nonumber |\partial_{\xi}^{l}P^d- \partial_{\xi}^{l}R^d|_{D_{\frac{7}{8} {\alpha}}^d} &\leq& c {\gamma}^{d+ m+ 5} {s^{m}} {\mu }^2,\\
   \nonumber |\partial_{\xi}^{l}R^d |_{D_{{\frac{7}{8} {\alpha}}}^d} &\leq&  c  {\gamma} ^{d+ m +5} {s^{m}} {\mu}.
 \end{eqnarray}

\subsubsection{Homology Equation}\label{homology equation}
Consider the following homology equation:
 \begin{eqnarray}\label{5Posion4}
  \{N^d,F^d\}+ R^d-[R^d]=0,
\end{eqnarray}
where \begin{eqnarray}
\nonumber F^d =\sum _{0< |\mathbf{k}|\leq K_+^b, |\imath|\leq m} f_{k_1k_2k_3\imath_1\imath_2\imath_3}^d {y^{\imath_1} \eta^{\imath_2} I^{\imath_3}} {e^{\sqrt{-1} \frac{\langle k_1, x\rangle}{\lambda_1}} e^{\sqrt{-1} \langle k_2, \theta \rangle} e^{\sqrt{-1}\lambda_2 \langle k_3, \varphi\rangle} },
 \end{eqnarray}
and $[R^d]= \int _{\mathbb T^d \times\mathbb T^d \times\mathbb T^d} R^d dx d\theta d\varphi$ is the average of truncation $R^d$. In view of (\ref{5Posion4}), comparing coefficients, we have
\begin{eqnarray*}
L_k^d f_{k_1k_2k_3\imath_1\imath_2\imath_3} = P_{k_1k_2k_3\imath_1\imath_2\imath_3},
\end{eqnarray*}
where
\begin{eqnarray*}
L_k^d &=& \frac{1}{\lambda_1 } \langle k_1, \omega+ \partial_y\langle\mathbf{b}, \mathfrak{A}^d \mathbf{b} \rangle+ \partial_y \hat{h}\rangle+ \langle k_2 , \Lambda+ \partial_\eta \langle\mathbf{b}, \mathfrak{A}^d \mathbf{b} \rangle +\partial_y \hat{h}  \rangle\\
&~&  + \lambda_2 \langle k_3, \Omega+ \partial_I \langle\mathbf{b}, \mathfrak{A}^d \mathbf{b} \rangle + \partial_I \hat{h}\rangle.
\end{eqnarray*}

Denote by $M^{*d} = \max\limits_{|l|\leq d, |i|< m+5}  \{\partial_{\xi}^l \partial_{y}^{i_1} \partial_{\eta}^{i_2}\partial_{I}^{i_3} h (y, \eta, I,\xi)\}$. With the assumptions
\begin{eqnarray}
\label{5b4} \max\limits_{|l|\leq d,|i|< m+5} |\partial_\xi^l \partial_y^{i_1}\partial_\eta^{i_2} \partial_I^{i_3} h^d - \partial_\xi^l \partial_y^{i_1}\partial_\eta^{i_2} \partial_I^{i_3} h_0^d|_{D^d(s)\times G_+^d} &\leq& \mu_0^{\frac{1}{2}},\\
\label{5c4} \frac{\gamma - \gamma_+}{(M^{*d} +1) {K_+^b}^{\tau + 1}}&>& 5s,
\end{eqnarray}
similarly, we have
\begin{eqnarray}
\nonumber&~& |\partial_\xi^l \partial_y^{i_1} \partial_\eta^{i_2}\partial_I^{i_3} f_{k_1k_2k_3\imath_1\imath_2\imath_3}^d|_{D^d(s) \times G_+^d}\\
  &\leq& \left\{
         \begin{array}{ll}
           |\tilde{\lambda}|^{-1} |k|^{(|l|+|i|+1) \tau +|l|+ |i|} \mu s^{m - |i|} e^{-\frac{\delta_1 k_1 r}{\lambda_1}} \\e^{-\delta_2 k_2 r} e^{-\lambda_2 \delta_1 k_3 r}, & \hbox{$|i| \leq m;$} \\
           |\tilde{\lambda}|^{-1} |k|^{(|l|+|i|+1) \tau +|l|+ |i|} \mu e^{-\frac{\delta_1 k_1 r}{\lambda_1}}\\e^{-\delta_2 k_2 r} e^{-\lambda_2 \delta_1 k_3 r}, & \hbox{$m< |i| \leq m+4,$}
         \end{array}
       \right.~~~~~~~
\end{eqnarray}
for $0 < |\mathbf{k}|\leq K_+^b$, $|l| \leq d$. Therefore,
\begin{eqnarray}\label{5118}
 |\partial_\xi^l \partial_x^{\kappa_1} \partial_y^{i_2} \partial_\theta^{\kappa_2} \partial_\eta^{i_2} \partial_\varphi^{\kappa_3} \partial_I^{i_3} F^d|  \leq \left\{
                                                                                                                                                     \begin{array}{ll}
                                                                                                                                                       c s^{m - |i|} \mu \Gamma^d(r - r_+), & \hbox{$|i| \leq m$;} \\
                                                                                                                                                       c \mu \Gamma^d(r - r_+), & \hbox{$m<|i| \leq m+4$,}
                                                                                                                                                     \end{array}
                                                                                                                                              \right.~~
\end{eqnarray}
where
\begin{eqnarray*}
\Gamma^d(r - r_+) &=&  \sum_{\substack{0<|\mathbf{k}|\leq K_+^b\\|i|,|\kappa|\leq m+4\\}} \frac{|\mathbf{k}|^{(|l|+|i|+1)\tau+ |l|+|i|+1}}{{\lambda_2}^{\delta_3}} |\frac{\delta_1 k_1}{\lambda_1}|^{\kappa_1} |\delta_2 k_2|^{\kappa_2} |\delta_3 \lambda_2 k_3|^{\kappa_3}\\&~&e^{-\frac{\delta_1k_1(r -r_+)}{8 \lambda_1}} e^{-\frac{\delta_2k_2 (r-r_+)}{8}} e^{-\frac{\lambda_2 \delta_3 k_3 (r -r_+)}{8}}.
\end{eqnarray*}
\subsubsection{Frequency Retention}\label{Frequency Retention1}
Under the time $1-$map $\Phi _{F^d}^1$ of the flow generated by a Hamiltonian $F^d$ and the transformation $\phi^d : x \rightarrow x,$ $y \rightarrow y+ y_{*}^d,$ $\theta \rightarrow \theta$, $\eta \rightarrow \eta+\eta_*^d$, $\varphi \rightarrow \varphi,$ $I\rightarrow I+I_*^d$, we have
\begin{eqnarray*}
H_+^d &=& H^d \circ \Phi_{F^d}^1 \circ \phi\\
  &=& e^d + \langle \mathbf{a}, \mathbf{b}+ \mathbf{b_*}\rangle+\frac{1}{2}\langle \mathbf{b} + \mathbf{b_*},
\mathfrak{A}^d  (\mathbf{b} + \mathbf{b_*}) \rangle  +\hat{h}^d(y+y_*^d, \eta+\eta_*^d, I+I_*^d)\\
&~&+ [R^d](y + y_*^d,\eta+ \eta_*^d, I+ I_*^d) +\bar{P}_+^d(x, y,\theta,\eta,\varphi,I,\xi)\\
 &=& e^d +\langle \mathbf{a}, \mathbf{b_*} \rangle  + \frac{1}{2}\langle \mathbf{b_*},
\mathfrak{A}^d \mathbf{b_*}\rangle+ \hat{h}(y_*^d, \eta_*^d, I_*^d)+ [R](y_*^d, \eta_*^d, I_*^d) + \langle \mathbf{a}, \mathbf{b} \rangle \\
  &~&+ \frac{1}{2} \langle \mathfrak{A}^d \mathbf{b_*},\mathbf{b} \rangle+ \langle\mathbf{\partial\hat{h}}, \mathbf{b} \rangle +\langle \mathbf{P}, \mathbf{b} \rangle +\frac{1}{2}\langle \mathbf{b},
\mathfrak{A}^d \mathbf{b}\rangle +\frac{1}{2}\langle \mathbf{b},  \hat{\mathfrak{h}} \mathbf{b}\rangle \\
&~& + \frac{1}{2}\langle \mathbf{b}, [\mathfrak{R}] \mathbf{b}\rangle+\hat{h}(y+ y_*^d, \eta+\eta_*^d,I+I_*^d)- \hat{h}(y_*^d,\eta_*^d,I_*^d)- \langle \mathbf{\partial \hat{\mathbf{h}}}, \mathbf{b} \rangle\\
&~&- \frac{1}{2}\langle\mathbf{b},  \hat{\mathfrak{h}} \mathbf{b}\rangle +[R](y+y_*^d, \eta+ \eta_*^d, I+I_*^d) - [R](y_*^d,\eta_*^d,I_*^d)\\
&~&  - \langle  \mathbf{\partial[R]}, \mathbf{b}\rangle-\frac{1}{2}\langle \mathbf{b},
 [\mathfrak{R}] \mathbf{b}\rangle -  \langle \mathbf{P},  \mathbf{b} \rangle + \langle \mathbf{\partial [R]}, \mathbf{b}\rangle + \bar{P}_+^d \circ \phi,
\end{eqnarray*}
where $ \bar{P}_+^d(\frac{x}{\lambda_1},y,\theta,\eta,\lambda_2 \varphi,I \xi) = \int_0^1 \{R_t^d, F^d\} \circ \Phi_{F^d}^t dt+ (P^d- R^d)\circ \Phi_{F^d}^1,$ $R_t^d = t R^d+(1 - t) [R^d].$

Let $\mathcal{A}^d$ be an $n \times n$ nonsingular minor of $\mathfrak{A}^d$. Then, for $ \mathfrak{A}^d$, there is a orthogonal matrix $T$ such that
\begin{eqnarray*}
 T^{-1}\mathfrak{ A}^d T = \left(
                                                                                                           \begin{array}{cc}
                                                                                                             \mathcal{A}^d & 0 \\
                                                                                                             0 & 0 \\
                                                                                                           \end{array}
                                                                                                         \right)
 .
\end{eqnarray*}
  Denote $\mathbf{b}^d = (y^d, \eta^d, I^d)^T$ and $\mathbf{p}^d = (p_{000100}^d, p_{000010}^d, p_{000001}^d)^T$, which are the vectors formed by the first $n$ components of $T^{-1}\mathbf{b}T$ and $T^{-1} \mathbf{P} T$, respectively, where the number of the components of $y^d$ and $p_{000100}^d$ are $n_1$, the number of the components of $\eta^d$ and $p_{000010}^d$ are $n_2$, the number of the components of $I^d$ and $p_{000001}^d$ are $n_3$, and $n = n_1+n_2+n_3$, and denote $ \hat{h}^d(y, \eta, I)= \hat{h}^d((y^d,0)^T, (\eta^d,0)^T, (I^d,0)^T).$ Then by the implicit function theorem, the equation
\begin{eqnarray}\label{5117}
\mathcal{A}^d \mathbf{b}^d +\mathbf{\partial_d \hat{h}} = \mathbf{p}^d
\end{eqnarray}
admits a unique solution $(\bar{y}_{*}^d, \bar{\eta}_*^d, \bar{I}_*^d)$ on $D(s)$, which also smoothly depends on $\xi$. Define $\mathbf{b}_{*}^d = (y_{*}^d, \eta_{*}^d, I_{*}^d)= ((\bar{y}_{*}^d,0 )^T, (\bar{\eta}_{*}^d,0 )^T, ( \bar{I}_{*}^d,0 )^T)$. By (\ref{5117}), we have
\begin{eqnarray*}
\mathfrak{A}^d \mathbf{b}_*^d+ \partial \mathbf{\hat{h}} =\mathbf{p_1},
\end{eqnarray*}
where $\mathbf{p_1} = ((p_{000100},0)^T, (p_{000010},0)^T, (p_{000001},0)^T)^T$.

Actually, the unique solution of (\ref{5117}), $(\bar{y}_{*}^d, \bar{\eta}_*^d, \bar{I}_*^d)$, is the translation of the transformation $\phi^d$, $(y_*^d, \eta_*^d, I_*^d)$. And to simplify the symbol we denote $(\bar{y}_{*}^d, \bar{\eta}_*^d, \bar{I}_*^d)$ by $(y_*^d, \eta_*^d, I_*^d)$. Then
 \begin{eqnarray*}
    H_+^d = N_+^d + P_+^d= e_{+}^d + \langle \mathbf{a}_+, \mathbf{b} \rangle + h_{+}^d(y, \eta, I)+P_+^d(x,y,\theta,\eta,\varphi,I,\xi),
 \end{eqnarray*}
where
\begin{eqnarray*}
 \nonumber e_+^d &=& {e}^d + \langle  \mathbf{a},\mathbf{b}_*  \rangle +\frac{1}{2} \langle\mathfrak{A}^d  \mathbf{b}_*,  \mathbf{b}_*\rangle + \hat{h}^d(y_{*}^d,\eta_{*}^d, I_{*}^d) + [R^d](y_{*}^d,\eta_{*}^d, I_{*}^d),\\
 \mathbf{a}_+ &=& \mathbf{a} + \mathbf{P}  - \mathbf{p_1},~~ \mathfrak{A}_+^d =\mathfrak{A}^d + \hat{\mathfrak{h}}^d + [\mathfrak{R}],~~h_+^d =  \frac{1}{2}\langle \mathbf{b}, \mathfrak{A}_+^d \mathbf{b}\rangle + \hat{h}_+^d,\\
\nonumber ~~ \\
\nonumber \hat{h}_+^d &=& \hat{h}^d(y+ y_{*}^d, \eta+\eta_*^d, I+I_*^d)- \hat{h}^d(y_{*}^d,\eta_{*}^d,I_{*}^d)- \langle \mathbf{\partial \hat{h}}, \mathbf{b} \rangle - \frac{1}{2}\langle \mathbf{b}, \hat{\mathfrak{h}} \mathbf{b}\rangle\\
&~&+ [R^d](y+y_*^d,\eta+\eta_*^d, I+I_*^d)- [R^d](y_*^d, \eta_*^d, I_*^d)- \langle\mathbf{ \partial [R]},  \mathbf{b}\rangle- \frac{1}{2}\langle \mathbf{b}, [\mathfrak{R}] \mathbf{b}\rangle,\\
\nonumber P_+^d &=& \bar{P}_+^d (x,y,\theta,\eta,\varphi,I,\xi) \circ \phi^d + \psi^d,\\
\psi^d &=& -  \langle \mathbf{P},  \mathbf{b}\rangle + \langle\mathbf{\partial [R]}, \mathbf{b}\rangle = \sum_{2\leq |J| \leq m, |J-1|\leq m+1} \left(
                                             \begin{array}{c}
                                               J \\
                                               1 \\
                                             \end{array}
                                           \right)
\langle \mathbf{c}, \mathbf{b}\rangle,
\end{eqnarray*}
where $\mathbf{c} = (P_{000100} {y_*^d}^{J-1},  P_{000010} {\eta_*^d}^{J-1}, P_{000001} {I_*^d}^{J-1})^T$.

\subsubsection{ Estimate on $N_+^d$} \label{5137}
Denote $M_{*}^d = \max\limits_{\xi \in G_0^d} |{\mathcal{A}^d}^{-1}(\xi)| +1$ and let $\mu_0$ small enough, say, $\mu_0 < \frac{1}{8{M_*^d}^2(M^{*d} +1)}$, such that $M_*^d(M^{*d} +1)s_0^2 < \frac{1}{4}$.
For $\xi \in G_+^d$, denote
\begin{eqnarray*}
B^d = \mathcal{A}^d + S \hat{h}.
\end{eqnarray*}
 Then, by (\ref{5117}),  $B^d \mathbf{b}_*^d = -\mathbf{p}^d.$  In the same way as subsection \ref{Estimation N_+}, under the assumption
\begin{eqnarray}\label{5d4}
    4  M_{*}^d (M^{*d}+1)\gamma ^{d+ m+5}s^{m-1} \mu < \frac{1}{2},
\end{eqnarray}
we have
 \begin{eqnarray}
\nonumber  |\partial_\xi^l \mathbf{b}_*^d|&<& c M_{*}^d \gamma^{d + m +5} s^{m - 1} \mu,\\
 \nonumber |\partial_{\xi}^{l}{e_+^d} - \partial_{\xi}^{l}e^d| _{G_+^d} &\leq& c \gamma ^{d+m+5} s^{m-1} \mu,\\
  \nonumber |\partial_{\xi}^{l}{\omega_+^d} - \partial_{\xi}^{l}{\omega^d}| _{G_+^d} &\leq& c \gamma ^{d+m+5} s^{m-1} \mu,\\
 \label{580} |\partial_y ^{i_1} \partial_\eta^{i_2} \partial_I^{i_3} \partial_{\xi}^{l}({h_+^d}(y) - {h^d}(y))| _{G_+^d} &\leq& \left\{
                                                                                                                       \begin{array}{ll}
                                                                                                                         c \gamma ^{d+m+5} s^{m - |i|}\mu, & \hbox{$|i| \leq m$;} \\
                                                                                                                         c \gamma ^{d+m+5}\mu, & \hbox{$m<|i| \leq m+4$.}
                                                                                                                       \end{array}
                                                                                                                     \right.~~~~~~~
 \end{eqnarray}

\subsubsection{Estimate on $\Phi_+^c$}\label{5138}
Denote $\Phi_+^d = \Phi_{F^d}^1 \circ \phi$ and
\begin{eqnarray}\label{5120}
   \Phi _{F^d}^t = id + \int _0 ^t X_{F^d} \circ \Phi_{F^d} ^{u} d u,
\end{eqnarray}
where $X_{F^d} = (F_y^d, -F_x^d, F_\eta^d, -F_\theta^d, F_I^d, -F_\varphi^d)^T$ denotes the vector field generated by $F^d$, the estimate of $\Phi_+^d$ is intimately tied to one of $X_{F^d}$, implying the essentiality for $X_{F^d}$.

By (\ref{5118}), $|\partial _\xi ^l F^d |_{D_{\frac{7}{8}\alpha}^d} \leq c s^{m} \mu \Gamma^d(r -r_+),$ and thus, by the Cauchy estimate on ${D_{\frac{3}{4}\alpha}^d}$,
\begin{eqnarray}\label{5119}
\nonumber (r - r_+) |\partial _{\xi}^l \partial_I F^d|,~~ (r - r_+) |\partial _{\xi}^l \partial _\eta F^d|, ~~(r - r_+) |\partial _{\xi}^l \partial _y F^d|,~ \\
s| \partial _{\xi}^l \partial _\varphi F^d|,~~s| \partial _{\xi}^l \partial _\theta F^d|, ~~s| \partial _{\xi}^l \partial _x F^d|  ~\leq ~ c s^{m} \mu \Gamma^d(r -r_+).
\end{eqnarray}
Then we have, inductively, $|D^n \partial _\xi ^l F^d| \leq c  \mu \Gamma^d(r -r_+),$ $n \leq 4.$

Denote $\Phi _{F^d}^t =(\phi_1^d,\phi_2^d,\phi_3^d,\phi_4^d,\phi_5^d,\phi_6^d)^T$, where $\phi_1^d$, $\phi_2^d$, $\phi_3^d$, $\phi_4^d$, $\phi_5^d$, $\phi_6^d$  are components of $\Phi_{F^d}^t$ in the directions of $x,$ $y,$ $ \theta$, $\eta$, $\varphi$, $I$, respectively. Let $(x, y,\theta,\eta,\varphi,I)$ be any point in $D_{\frac{1}{4}\alpha}^d$ and let $t_{*} = sup \{ t\in [0,1] : \Phi_{F^d}^t(x,y )\in D_\alpha^d \}$. We note that $D_\alpha^d \subset \hat{D}^d(s).$
By (\ref{5120}), we have
\begin{eqnarray*}
  |\phi_1^d(x, y) - x| &\leq &\int_0^t |F_y ^d\circ \Phi_{F^d}^{u}|_{D_\alpha^d} du \leq|F_y^d| _{\hat{D}^d(s)}< \frac{1}{8} (r - r_+),\\
   |\phi_2^d (x, y) - y| &\leq & \int_0^t |F_x^d \circ \Phi_{F^d}^{u}|_{D_\alpha^d} du  \leq|F_x^d| _{\hat{D}^d(s)}<  \frac{1}{8} \alpha,\\
  |\phi_3^d(x, y) - \theta|&\leq & \int_0^t |F_\eta ^d\circ \Phi_{F^d}^{u}|_{D_\alpha^d} du \leq|F_\eta^d| _{\hat{D}^d(s)}< \frac{1}{8} (r - r_+),\\
   |\phi_4^d (x, y) - \eta|&\leq &\int_0^t |F_\theta^d \circ \Phi_{F^d}^{u}|_{D_\alpha^d} du  \leq|F_\theta^d| _{\hat{D}^d(s)}< \frac{1}{8} \alpha,\\
  |\phi_5^d(x, y) - \varphi|&\leq &\int_0^t |F_I ^d\circ \Phi_{F^d}^{u}|_{D_\alpha^d} du \leq|F_I^d| _{\hat{D}^d(s)}< \frac{1}{8} (r - r_+),\\
  |\phi_6^d (x, y) - I|&\leq &\int_0^t |F_\varphi^d \circ \Phi_{F^d}^{u}|_{D_\alpha^d} du  \leq|F_\varphi^d| _{\hat{D}^d(s)}<  \frac{1}{8} \alpha,
\end{eqnarray*}
provided
\begin{eqnarray}
\label{5f4}  c \Gamma^d(r-r_+) \mu&<& \frac{1}{8} (r - r_+),\\
\label{5g4} c \mu s^{m} \Gamma^d(r - r_+) &<& \frac{1}{8} \alpha s.
\end{eqnarray}
Then $|\phi_1^d|, |\phi_3^d|, |\phi_5^d|< r_+ + \frac{3}{8} ( r- r_+)$ and $|\phi_2^d|, |\phi_4^d|, |\phi_6^d| < \frac{3}{8} \alpha s.$ Therefore, $\Phi _{F^d}^t : D_{\frac{1}{4}\alpha}^d \rightarrow D_{\frac{1}{2} \alpha}^d \subset D_\alpha^d$. Hence $t_* = 1$ and $\Phi_{F^d}^t: D_{\frac{\alpha}{4}}^d \rightarrow D_{\frac{\alpha}{2}}^d$. With assumption
\begin{eqnarray}\label{5e4}
c s^{m-1} \mu < \frac{1}{8} \alpha s,
\end{eqnarray}
we have $\phi^d: D_{\frac{1}{8} \alpha}^d \rightarrow D_{\frac{1}{4}\alpha}^d.$

With the standard Whitney extension theorem, it is easy to see that $F^d$ and $y_{*}^d$ can be extended to functions of H$\ddot{o}$lder class $C^{m+3,d-1+\sigma_0}(\hat{D}^d(\beta_0)\times {G_0^d})$, respectively, where $0< \sigma_0 <1$ is fixed. Moreover,
\begin{eqnarray}
\nonumber \|F^d\|_{{C^{m+3, d-1+\sigma_0}}(\hat{D}^d(\beta_0) \times G_0^d)} &\leq& \mu \Gamma^d(r - r_+),\\
\nonumber \|y_{*}^d\| _{C^{d-1 + \sigma_0 }(G_0^d)}&\leq& c \mu \Gamma^d(r- r_+).
\end{eqnarray}

Thus $\Phi_+^d: D_+^d \rightarrow D_{\frac{1}{2}\alpha}^d$ is well defined, symplectic and real analytic for all $ \xi \in G_+^d$. It is easy to see that $\Phi_+^d$ maps $\hat{D}_+^d$ into $D^d(r,\beta)$ for all $\xi \in G_0^d$. We note that
\begin{center}
$\Phi _{F^d}^t = id + \int_0^t X_{F^d} \circ \Phi_{F^d}^u du$,~~~~~~~~~~$0\leq t \leq 1$,\\
$\|X_{F^d}\| _{C^{m+2, d-1+\sigma_0}(\hat{D}^d(\beta _0)\times G_0^d)} \leq c \|F^d\|_{C^{m+3, d-1+\sigma_0}(\hat{D}^d(\beta _0)\times G_0^d)}.$
\end{center}
Supposing
\begin{eqnarray}
\label{5h3}  c \mu \Gamma^d (r- r_+) &<& \frac{1}{8} (r - r_+),\\
\label{5i3} c \mu \Gamma ^d (r- r_+)  &<& \beta - \beta_+,
\end{eqnarray}
and applying the Gronwall inequality and the definition of $\Phi_{F^d}^t$, inductively, we have that on $\tilde{D}_+^d \times G _0^d$,
\begin{eqnarray}\label{574}
  |\Phi_{F^d}^t - id |, |\partial _y \Phi _{F^d}^t - I_{2d}|, |\partial_y ^j \Phi _{F^d}^t| \leq c\mu \Gamma^d (r-r_+).
\end{eqnarray}
Then $\Phi _+^d = \Phi_{F^d}^1 \circ \phi^d: \hat{D}_+^d \rightarrow D^d(r, \beta)$ is of classes $C^{m+2}$ and also depends $C^{d-1+ \sigma_0}$ smoothly on $\xi \in G_0^d$, where $\sigma_0 $ is described as above in Section \ref{couple}. Moreover, $\| \Phi_+^d - id\|_{C^{m+2,d-1+\sigma_0}(\tilde{D}_+^d \times G_0^d)}\leq c \mu \Gamma^d (r- r_+).$

 Thus, under the symplectic transformation $\Phi _+^d = \Phi_{F^d} ^ 1\circ \phi^d,$ the new Hamiltonian reads
\begin{center}
   $ H^d\circ \Phi _+^d = N_+^d +P_+^d$,
\end{center}
where
\begin{eqnarray*}
    N_+^c &=& e_+^d + \langle \mathbf{a}_+ , \mathbf{b}\rangle + h_+^d,~~~~~~P_+ ^d = \bar{P}_+^d \circ \phi^d + \psi^d,\\
\mathbf{a}_+  &=&  \mathbf{a} + \mathbf{P}
     + \mathfrak{A}^d \mathbf{b}_*^d
      + \mathbf{\partial \hat{h}},
\end{eqnarray*}
and $h_+ ^d$, $\mathfrak{A}^d$, $\psi^d$ and $\hat{h}^d$ have the same forms as above.  Furthermore, under the assumptions mentioned in this section, it is easy to check that for $0< |\mathbf{k}| \leq K_+^b,~\xi \in G_+^d$
\begin{eqnarray*}
  |\frac{\delta_1}{\lambda_1}\langle k_1, \omega_+ \rangle+ \delta_2 \langle k_2, \Lambda_+ \rangle+\delta_3 \lambda_2 \langle k_3, \Omega_+ \rangle)| > \frac{|\tilde{\lambda}|\gamma}{|\mathbf{k}|^{\tau}}.
\end{eqnarray*}

\subsubsection{Estimate on $P_+^d$}
Note
\begin{eqnarray*}
P_+^d = \bar{P}_+^d \circ \phi^d+ \psi^d = \left(\int_0^1 \{R_t^d , F^d\}\circ \Phi _{F^d}^t dt + (P^d-R^d)\circ \Phi _{F^d}^1 \right )\circ \phi^d + \psi^d.
\end{eqnarray*}
By the above estimates, we see that, for all $|l| \leq d$, $0\leq t \leq 1$,
\begin{eqnarray}
   \nonumber  |\partial _\xi^ l \{R_t^d, F^d\}\circ \Phi_{F^d}^t|_{{D_{\frac{1}{4}\alpha}^d}\times G_+^d }&\leq& c  \gamma^{d+m+5} s^{m} {\mu}^2 \Gamma^d (r- r_+),\\
   \nonumber |\partial _\xi^l(P^d -R^d) \circ \Phi_{F^d}^1|_{{D_{\frac{1}{4}\alpha}^d}\times G_+^d }&\leq&c \gamma^{d+m+5} s^{m} {\mu}^2,\\
   \nonumber|\partial_\xi^l \psi^d|~~~~~&\leq& c \gamma^{d+m+5} s^{m} {\mu}^2.
\end{eqnarray}
Moreover, we have
\begin{eqnarray*}
|\partial_\xi^l \phi^d |_{D_+^d \times G_+^d} \leq c \gamma^{d +m +5} s^{m-1} \mu ~for ~|l| \leq d.
\end{eqnarray*}
 Hence, by the definition of $P_+^d$,
\begin{center}
  $|\partial _\xi^l P_+^d |_{D_+^d \times G _+^d} \leq c\gamma^{d+m+5} s^{m} {\mu}^2 (\Gamma^d (r- r_+) +2 )$, ~~~~~~$|l|\leq d$.
\end{center}

Let $c_0$ be the maximal one of the $c's$ we mentioned above in this section and define $\mu_+ = 8^m c_0 \mu^{1 + \sigma}.$ With the assumption
\begin{eqnarray}\label{5j4}
 \mu^{\sigma } (\Gamma ^d(r- r_+) +2) \leq \frac{\gamma_{+}^{d+m+5}}{\gamma^{d+m+5}},
 ~~~on~ D_{+}^d \times G_+^d,
\end{eqnarray}
we have
\begin{eqnarray}
\nonumber |\partial _{\xi}^l P_+^d| &\leq& 8^m c_0 s^{m} {\mu}^{1+ \sigma} {\mu}^{1-2\sigma} {\mu}^{\sigma} {\gamma}^{d+m+5} (\Gamma^d(r- r_+) +2)\\
 \nonumber &\leq& c_0\gamma_{+}^{d+m+5} s_{+}^{m} \mu _+,~~~~~~~~|l|\leq d.
\end{eqnarray}

We now complete one KAM step.

\subsection{Iteration Lemma}
Let $r_0$, $s_0$, $\gamma _0$, $\beta_0$, $\mu_0$, $N_0^d$, $e_0^d$, $\omega_0^d$, $h_0^d$, $\mathfrak{A}_0^d$, $\hat{h}_0^d$, $P_0^d$ be given as above. And let $\hat{D}_0^d = D^d(r_0, \beta_0)$. For (\ref{5115}) we have:
\begin{lemma}\label{balala}
If (\ref{5116}) holds for a sufficiently small $\mu= \mu (r, s, d, \tau)$, then the KAM step described in subsection \ref{KAM step3} is valid for all $\nu =0,1,\cdots,$ and sequences
$G_\nu^d, H_\nu^d, N_\nu^d, e_\nu^d, \omega_\nu^d, h_\nu^d, \mathfrak{A}_\nu^d, \hat{h}_\nu^d, P_\nu^d, \Phi_\nu^d$, $\nu = 1,2,\cdots ,$ possess the following properties:
\begin{itemize}
\item[{\bf (1)}]$\Phi _\nu^d : \hat{D}^d \times G_0^d \rightarrow \hat{ D}_{\nu-1}^d$, $D_\nu^d \times G_\nu^d \rightarrow D_{\nu-1}^d$ is symplectic for each $\xi\in G_0^d$
 or $G_\nu^d$, and is of class $C^{m+2,d-1+\sigma_0}$, $C^{\iota, d}$, respectively, where $\iota$ stands for real analyticity, $0<\sigma_0<1$ is fixed, and
 \begin{eqnarray}\label{575}
 \|\Phi_\nu^d - id\| _{C^{m+2,d-1+\sigma_0}(\hat{D}_\nu^d \times G^d)} \leq \frac{\mu^{\frac{1}{2}}}{2^{\nu}}.
 \end{eqnarray}
Moreover, on $\hat{D}_\nu^d \times G^d$,
\begin{center}
  $H_\nu^d= H_{\nu-1}^d\circ \Phi_{\nu}^d= N_{\nu}^d+ P_\nu^d,$
  \end{center}
  where
  $ N_\nu^d = e_\nu^d + \langle \mathbf{a}_\nu , \mathbf{b}\rangle + \frac{1}{2} \langle \mathbf{b}, \mathfrak{A}_\nu^d \mathbf{b}\rangle + \hat{h}^d(y),$  $\mathfrak{A}_\nu^d$ has an $n \times n $ minor $\mathcal{A}_{\nu}^d$, which is nonsingular on $G_{\nu}^d$, $ \hat{h}^d$ are all terms with the form of $y^{\iota_1}\eta^{\iota_2}I^{\iota_3}$, $|\iota_1|+|\iota_2|+|\iota_3| \geq 3$, in $N_{\nu}^d$;
\item[{\bf (2)}]Under assumption $\bf{(A1)}$, we have $(\mathbf{a}_\nu)_q=(\mathbf{a}_\nu)_q $, $\forall$ $\xi \in G_\nu^d,$ $q= 1,2,\cdots,n;$
\item[{\bf(3)}]For all $|l|\leq d ,$
  \begin{eqnarray}
   \label{576} |\partial_\xi^l e_\nu^d -\partial_\xi^l e_{\nu -1}^d | _{G_\nu^d} &\leq & \gamma _0^{d+m+4} \frac{\mu}{2^{\nu}};\\
  \label{577}  |\partial_\xi^l e_\nu^d -\partial_\xi^l e_0^d | _{G_\nu^d} &\leq&  \gamma _0^{d+m+4} \mu;\\
  \label{578} |\partial_\xi^l {\omega}_{\nu}^d -\partial_\xi^l {\omega}_{\nu -1}^d |_{G_\nu^d} &\leq&  \gamma _0^{d+m+4} \frac{\mu}{2^{\nu}};\\
   \label{579} |\partial_\xi^l {\omega}_\nu^d -\partial_\xi^l {\omega}_0^d |_{G_\nu^d} &\leq&  \gamma _0^{d+m+4} \mu;\\
 \label{581} |\partial _\xi^l h_\nu^d  -\partial _\xi^l  h_{\nu-1}^d |_{D_\nu^d \times G_\nu^d} &\leq&  \gamma _0 ^{d+m+4} \frac{{\mu}^{\frac{1}{2}}}{2^{\nu}};\\
 \label{582}  |\partial _\xi^l h_\nu^d  -\partial _\xi^l h_0^d|_{D_\nu^d \times G_\nu^d} &\leq&  \gamma _0 ^{d+m+4} {\mu}^{\frac{1}{2}};\\
  \label{583}  |\partial_{\xi}^l P_\nu^d|_{D_\nu^d \times G_\nu^d}&\leq&  \gamma_ \nu ^{d+m+5} s_\nu^m {\mu_\nu};
 \end{eqnarray}
 \item[{\bf(4)}]  $G_{\nu+1}^d = \left \{\xi \in G_\nu^d: |\frac{\delta_1}{\lambda_1}\langle k_1, \omega \rangle+ \delta_2 \langle k_2, \Lambda \rangle+\delta_3 \lambda_2 \langle k_3, \Omega \rangle)| > \frac{|\tilde{\lambda}|\gamma}{|\mathbf{k}|^{\tau}},\quad for~ all ~~0<|\mathbf{k}|\leq K_{\nu+1}^b \right\}.$
\end{itemize}
\end{lemma}
\begin{proof}
Actually, it suffices to verify the assumptions that we put forward in Section \ref{couple} for all $\nu$. With the proofs of Lemmas \ref{lemma1} and \ref{lemma21}, the proofs of assumptions (\ref{5d4'''}), (\ref{5d4'}), (\ref{5d4''}), (\ref{5b4}), (\ref{5c4}), (\ref{5d4}), (\ref{5e4}) are easy. For the simplicity of the proof, we omit them and only show something complex.

Since the value of $\Gamma^d$,
\begin{eqnarray*}
\Gamma^d &=&  \sum_{\substack{0<|\mathbf{k}|\leq K_+^b\\|i|, |\kappa| \leq m+4}} \frac{|\mathbf{k}|^{(|l|+|i|+1)\tau+ |l|+|i|+1}}{\lambda_2^{\delta_3}} |\frac{\delta_1 k_1}{\lambda_1}|^{\kappa_1}|\delta_2 k_2|^{\kappa_2} |\delta_3 \lambda_2 k_3|^{\kappa_3}\\&~&~ e^{-\frac{\delta_1k_1(r -r_+)}{8 \lambda_1}} e^{-\frac{\delta_2k_2 (r-r_+)}{8}} e^{-\frac{\lambda_2 \delta_3 k_3 (r -r_+)}{8}}\\
&\leq& \sum_{\substack{0<|\mathbf{k}| \\|i|, |\kappa|\leq m+4}} \frac{|\mathbf{k}|^{(|l|+|i|+1)\tau+ |l|+|i|+1}}{\lambda_2^{\delta_3}} |\frac{\delta_1 k_1}{\lambda_1}|^{\kappa_1}|\delta_2 k_2|^{\kappa_2} |\delta_3 \lambda_2 k_3|^{\kappa_3}\\&~&~~ e^{-\frac{\delta_1k_1(r -r_+)}{8 \lambda_1}} e^{-\frac{\delta_2k_2 (r-r_+)}{8}} e^{-\frac{\lambda_2 \delta_3 k_3 (r -r_+)}{8}},
\end{eqnarray*}
is controlled by the value of $k_1$, $k_2$ and $k_3$, we divide the calculation of $\Gamma^d(r-r_+)$ into the following seven cases:
\begin{enumerate}
 \label{de} \item $|k_1| = |k_2|= 0, |k_3| \geq 1$,
  \item $|k_1| = |k_3|= 0, |k_2| \geq 1$,
  \item $|k_2| = |k_3|= 0, |k_1| \geq 1$,
  \item $|k_1|= 0, |k_2|, |k_3| \geq 1$,
  \item $|k_2|= 0, |k_1|, |k_3| \geq 1$,
  \item $|k_3|= 0, |k_1|, |k_2| \geq 1$,
  \item $|k_1|, |k_2|, |k_3| \geq 1$.
\end{enumerate}

According to the proofs of Lemmas \ref{lemma1} and \ref{lemma21} the proofs of (\ref{5j4}) in cases (1), (2) and (3) are obvious, and we omit the details.
In case (4),
\begin{eqnarray*}
\Gamma^d &=& \sum_{\substack{|k_2|\geq 1, |k_3|\geq 1\\|i|, |\kappa|\leq m+4}} \frac{|k|^{(|l|+ |i| +1 )\tau + |l|+ |i|+1} }{\lambda_2}|k_2|^{\kappa_2}|\lambda_2 k_3|^{\kappa_3}  e^{-\frac{k_2(r -r_+)}{8}} e^{- \frac{\lambda_2 k_3(r -r_+)}{8}}\\
&\leq& \sum_{\substack{|k_2|\geq 1, |k_3|\geq 1\\|i|, |\kappa|\leq m+4}} \frac{(|k_1 k_2|)^{(|l|+ |i| +1 )\tau + |l|+ |i|+1} }{\lambda_2}|k_2|^{\kappa_2}|\lambda_2 k_3|^{\kappa_3}  e^{-\frac{k_2(r -r_+)}{8}} e^{- \frac{\lambda_2 k_3(r -r_+)}{8}}\\
&\leq& \int_1^\infty t_1^{(|l|+ |i|+1)\tau +|l|+ |i|+|\kappa_2|+ |\kappa_3|+1} e^{- \frac{t_1 (r -r_+)}{8}} d t_1\\
&~&~\lambda_2^{\kappa_3 -1}\int_1^\infty t_2^{(|l|+ |i|+1)\tau +|l|+ |i|+|\kappa_2|+ |\kappa_3|+1} e^{- \frac{t_2 \lambda_2 (r -r_+)}{8}} d t_2\\
&\leq&\frac{ 2^{2(\nu+6){((|l|+ |i|+1)\tau +|l|+ |i|+|\kappa_2|+ |\kappa_3|+1)}}}{\lambda_2^{((|l|+ |i|+1)\tau +|l|+ |i|+|\kappa_2|+ 4)}}e^{\frac{1}{2^{\nu+6}}} e^{- \frac{\lambda_2}{2^{\nu+6}}} \\
&~&~({(|l|+ |i|+1)\tau +|l|+ |i|+|\kappa_2|+ |\kappa_3|+1)}!.
\end{eqnarray*}
When $\beta \in (0, \frac{\sigma^2}{2((|l|+ |i|+1)\tau +|l|+ |i|+|\kappa_2|+ 5)})$, it is easy to check that (\ref{5j4}) in case (4) holds.

In case (5), in the same way we have
\begin{eqnarray*}
\Gamma^d &=&\sum_{\substack{|k_1|\geq 1, |k_3|\geq 1\\|i|, |\kappa|\leq m+4}} \frac{|\mathbf{k}|^{(|l|+ |i|+1)\tau +|l|+ |i|+1}}{\lambda_2} (\frac{\delta_1 k_1}{\lambda_1})^{\kappa_3}  (\delta_3 \lambda_2 k_3)^{\kappa_3}\\
&~&~  e^{-\frac{k_1(r -r_+)}{8 \lambda_1}} e^{- \frac{\lambda_2 k_3(r -r_+)}{8}}\\
&\leq& c e^{(\kappa_1 -1) (ln (\kappa_1 -1) 2^{\nu+6} -1)}\frac{2^{2(\nu+6) ((|l|+|i|+1)\tau +|l|+|i|+|\kappa_1|+|\kappa_3| +1)}}{\lambda_2^{(|l|+|i|+1)\tau +|l|+|i|+|\kappa_1|+4}}.
\end{eqnarray*}
Obviously, when $\beta \in (0, \frac{\sigma^2}{2((|l|+|i|+1)\tau +|l|+|i|+|\kappa_1|+5)})$, (\ref{5j4}) in case (5) holds.

 In case (6),
\begin{eqnarray*}
\Gamma &=& \sum_{\substack{|k_1|\geq 1, |k_2|\geq 1\\|i|, |\kappa|\leq m+4}} |\mathbf{k}|^{(|l|+|i|+1)\tau+ |l|+|i|+1}  |\frac{\delta_1 k_1}{\lambda_1}|^{\kappa_1} |\delta_2 k_2|^{\kappa_2}\\
&~&~e^{-\frac{\delta_1k_1(r -r_+)}{8 \lambda_1}} e^{-\frac{\delta_2k_2 (r-r_+)}{8}}\\
&\leq&\int_1^\infty |\lambda_1|^{-\kappa_1} |t_1|^{(|l|+|i|+1)\tau+ |l|+|i|+ |\kappa_1|+|\kappa_2|+1} e^{- \frac{t_1(r-r_+)}{8\lambda_1}} d t_1\\
&~&~~\int_1^\infty |t_2|^{(|l|+|i|+1)\tau+ |l|+|i|+ |\kappa_1|+|\kappa_2|+1} e^{- \frac{t_2(r-r_+)}{8\lambda_1}} d t_2\\
&\leq& e^{(|\kappa_1 - 1|)(ln (\kappa_1 - 1) 2^{\nu+6} - 1)} 2^{2(\nu+6)({(|l|+|i|+1)\tau+ |l|+|i|+ |\kappa_1|+|\kappa_2|+1})} \\ &~&~(((|l|+|i|+1)\tau+ |l|+|i|+ |\kappa_1|+|\kappa_2|+1)!)^2 e^{-\frac{1}{2^{\nu+6}}}.
\end{eqnarray*}
Then (\ref{5j4}) holds for case (6).

 In case (7),
\begin{eqnarray*}
\Gamma &=& \sum_{|k_1|,|k_2|,|k_3| \geq 1} \frac{|\mathbf{k}|^{(|l|+|i|+1)\tau+ |l|+|i|+1}}{\lambda_2^{\delta_3}} |\frac{\delta_1 k_1}{\lambda_1}|^{\kappa_1}|\delta_2 k_2|^{\kappa_2} |\delta_3 \lambda_2 k_3|^{\kappa_3}\\
&~&~e^{-\frac{\delta_1k_1(r -r_+)}{8 \lambda_1}} e^{-\frac{\delta_2k_2 (r-r_+)}{8}} e^{-\frac{\lambda_2 \delta_3 k_3 (r -r_+)}{8}}\\
&\leq&\int_1^\infty |\lambda_1|^{- \kappa_1 } |t_1|^{(|l|+|i|+1)\tau+ |l|+|i| +|\kappa|+1} e^{-\frac{t_1(r -r_+)}{8 \lambda_1} d t_1}\\
&~&~\int_1^\infty  |t_2|^{(|l|+|i|+1)\tau+ |l|+|i| +|\kappa|} e^{-\frac{t_2(r -r_+)}{8} d t_2}\\
&~&~\int_1^\infty |\lambda_2|^{\kappa_3 -1 } |t_3|^{(|l|+|i|+1)\tau+ |l|+|i| +|\kappa|+1} e^{-\frac{t_3(r -r_+)\lambda_2}{8} d t_3}\\
&\leq&\frac{ \lambda_1^{-\kappa_1 +2} 2^{3(\nu+6) ((|l|+|i|+1)\tau+ |l|+|i| +|\kappa|+1)}}{\lambda_2^{(|l|+|i|+1)\tau+ |l|+|i| +|\kappa_1|+|\kappa_2|+4}}e^{\frac{1}{\lambda_1 2^{\nu+6}}}e^{\frac{1}{2^{\nu+6}}}e^{\frac{\lambda_2}{ 2^{\nu+6}}}\\
&~&~(((|l|+|i|+1)\tau+ |l|+|i| +|\kappa|+1)!)^3,
\end{eqnarray*}
which means that (\ref{5j4}) holds for $\beta \in (0, \frac{\sigma^2}{3[(|l|+|i|+1)\tau+ |l|+|i|+|\kappa_1|+|\kappa_2|+5]})$.
Hence the seven cases hold for $\beta \in (0, \frac{\sigma^2}{3[(d+m+5)\tau + d+ 2m +13]})$, i.e., (\ref{5j4}) holds for $\beta \in (0, \frac{\sigma^2}{3[(d+m+5)\tau + d+ 2m +13]})$. Analogously, (\ref{5f4}) holds.

For assumption (\ref{5g4}) we have
\begin{eqnarray*}
(1+ \sigma)^\nu (m +(1+m)\sigma -1-\sigma) -m -\frac{\sigma^2 m }{3} - \frac{\sigma^2}{3} +1 \geq 0,
\end{eqnarray*}
which means
\begin{eqnarray*}
\mu_0^{\frac{m (1+\sigma)^\nu - m}{(m+1)\sigma}} \mu_0^{(1+\sigma)^\nu} \varepsilon^{-\beta [(|l|+|i|+1)\tau+ |l|+|i| +|\kappa_1|+|\kappa_2|+4]}
\leq\mu_0^{\frac{(1+\sigma)^{\nu+1} -1}{\sigma(m+1)}}.
\end{eqnarray*}

\end{proof}

The convergences and measure estimates are similar to subsection \ref{5107}.

\section{Proof of (3) in Theorem \ref{dingli15}}\label{5108}
In this section, we sketch the proof of Theorem \ref{dingli15}. In fact, we can complete it by combining Section \ref{couple} and the arguments in \cite{LLL}, but the only difference is to solve the following equation on the frequency ratio instead of (\ref{5117}):
\begin{eqnarray*}
(\mathcal{A}^d +  {S}\hat{h}) {\mathbf{b}^d} - t_*^d (\mathbf{a}_{i_1}, \cdots, \mathbf{a}_{i_n})^T = -{\mathbf{p}^d},\\
 \langle  (\mathbf{a}_{i_1}, \cdots, \mathbf{a}_{i_n})^T, \mathbf{b}^d \rangle +\frac{1}{2} \langle\mathfrak{A}^d  \mathbf{b}_*,  \mathbf{b}_*\rangle + \hat{h}^d(y_{*}^d,\eta_{*}^d, I_{*}^d) + [R^d](y_{*}^d,\eta_{*}^d, I_{*}^d) = 0,
\end{eqnarray*}
where $(\mathbf{a}_{i_1}, \cdots, \mathbf{a}_{i_n})$ is the first $n$ components of $T^{-1} \mathbf{a}T$, which, by subisoenergetic nondegenerate condition $\bf{(A1^{\prime})}$ and implicit function theorem, admits a local smooth solution $(\mathbf{b}_*, t_*^d)$, $\mathbf{b}_* \in G^d$, $t_*^d \in R^1$, such that $\mathbf{b}_{j,*} = 0$ if $j\notin \{i_1, \cdots, i_n\}.$

\section{Application to Weakly Coupled N-Oscillators}\setcounter{equation}{0}\label{application}
In this section, we show an application of our results to the weakly coupled N-oscillators with quasiperiodic force, i.e., the following equation
\begin{eqnarray}\label{5109}
\ddot{x}_i + \nabla_i V(x_i) + \varepsilon (2 x_i - x_{i+1} - x_{i-1}) = \varepsilon \sin \frac{\omega_it}{\varepsilon^{\alpha}} + \varepsilon \sin \varepsilon^{\beta} \Omega_it,
\end{eqnarray}
where $i = 1, \cdots, N$, $0\leq\alpha$, $0\leq\beta\ll1,$ and $\nabla_i$ denote the derivation with respect to $x_i$, $\varepsilon$ is a small parameter.

The Hamiltonian function of system (\ref{5109}) is the following:
\begin{eqnarray}\label{5110}
\nonumber H (x_i, p_i, \eta,\theta) &=& \langle \omega, \eta \rangle + \langle \Omega, y\rangle+ \sum_{i =1}^N \frac{1}{2} p_i^2 + \sum_{i = 1}^N V(x_i)+ \varepsilon \sum_{i =1}^{N-1} \frac{1}{2}(x_{i+1} - x_i)^2 \\
&~& - \sum_{i =1}^N \varepsilon x_i \sin{\frac{\theta_i}{\varepsilon^{\alpha}}} - \sum_{i = 1}^{N} \varepsilon x_i \sin \varepsilon^{\beta} \psi_i,
\end{eqnarray}
where $\omega= (\omega_1, \cdots, \omega_N)$, $\Omega=(\Omega_1, \cdots,\Omega_N)$, $\eta= (\eta_1, \cdots ,\eta_N),$ $y=(y_1, \cdots, y_N)$  $\in \mathbb R^N$, $\theta_i = \omega_i t$, $\psi_i = \Omega_i t,$  $p_i  = \dot{x}_i$.

For Hamiltonian system (\ref{5110}), we have the following assumption:
\begin{itemize}
\item[{\bf (H3)}] There is some compact and connected subset of the $x_i - p_i$ plane, in which the level sets $H (x_i, p_i) = \frac{1}{2} p_i^2 + V(x_i) = h_i$, $i = 1, \cdots, N$,  denotes a closed curve, called $\Gamma(h_i)$, which encloses the origin $(0,0)$.
\item[{\bf (H4)}] Let $I_i = I_i(h_i)$ be the area enclosed by the closed curve $\Gamma(h_i)$, i.e.,
\begin{eqnarray*}
\oint_{\frac{1}{2} p_i^2 + V(x_i) = h_i^0(I)} p_i d x_i = I_i.
\end{eqnarray*}
\item[{\bf (H5)}] $\omega = \omega(I)$, $\Omega = \Omega(I)$.
\item[{\bf(H6)}]Denote $A = {\rm diag} (\partial_{I_1}^2 h_1^0, \cdots, \partial_{I_N}^2 h_N^0)$. $A$ has an $n\times n$ nonsingular minor $\mathcal{A}$.
\end{itemize}
Denote the standard symplectic transformation $\Psi:$ $(x_i, p_i) \mapsto (\varphi_i, I_i),$ which is given by
\begin{eqnarray}\label{5133}
S_{x_i} (x_i, I_i) = p_i, ~~~~S_{I_i} (x_i, I_i) = \varphi_i,
\end{eqnarray}
where $S(x_i, I_i) = \int_{\Gamma^*} y dx,$ and $\Gamma^*$ is the part of the enclosed curve $\frac{1}{2}p_i^2 + V(x_i) = h_i^0(I)$ connecting the $p_i-$axis with point $(x_i, p_i)$, oriented clockwise.

Under standard symplectic transformation (\ref{5133}), Hamiltonian system (\ref{5110}) reads
\begin{eqnarray}\label{5146}
\nonumber H&=&\langle \omega, \eta \rangle+\langle \Omega, y\rangle + \sum_{i =1}^N h_i^0 (I_i)- \varepsilon \sum_{i =1}^{N-1} \frac{1}{2}(x_{i+1} - x_i)^2- \sum_{i =1}^N \varepsilon x_i \sin{\frac{\theta_i}{\varepsilon^{\alpha}}}\\
\nonumber &~&- \sum_{i = 1}^{N} \varepsilon x_i \sin \varepsilon^{\beta} \psi_i\\
 &=& \langle \omega, \eta \rangle+\langle \Omega, y\rangle + \langle \Lambda, I\rangle + \frac{1}{2} \langle I, AI\rangle + \hat{h}(I) +\varepsilon P(I,\varphi, \frac{\theta}{\varepsilon^{\alpha}}, \varepsilon^{\beta}\psi),
\end{eqnarray}
which is a special case of (\ref{5131}), where $\Lambda = (\partial_{I_1} h_1^0, \cdots, \partial_{I_N} h_N^0)$, $A = {\rm diag} (\partial_{I_1}^2 h_1^0, \cdots, \partial_{I_N}^2 h_N^0),$ $\hat{h}(I) = O(I^3),$ $I= (I_1, \cdots, I_N)$, $\omega$, $\Omega,$ $\eta$, $y,$ $\theta = (\theta_1,\cdots,\theta_N),$ $\varphi = (\varphi_1, \cdots, \varphi_N)$, $\psi= (\psi_1, \cdots, \psi_N)$ $\in \mathbb R^N.$


We assume the following isoenergetic nondegeneracy:
\begin{itemize}
\item[{\bf (H7)}] Denote $\mathcal{M}$ by a given energy surface. There is a smoothly varying $n\times n$ nonsingular minor $\mathcal{A}(I)$ of $A(I)$ on $\mathcal{M}$, such that
$$\det \left(
               \begin{array}{cc}
                 \mathcal{A}(I) & \Lambda^*(I) \\
                 \Lambda^*(I)^T & 0 \\
               \end{array}
             \right) \neq 0$$
on $\mathcal{M}$, where $\Lambda^*(I) = (\partial_{I_{i_1}} h_{i_1}^0, \cdots, \partial_{I_{i_n}} h_{i_n}^0)$ and $i_1,$ $\cdots,$ $i_n$ denote the row indices of $\mathcal{A}$ in $A$.
\end{itemize}

\begin{remark}
The assumption ${\bf{(H7)}}$ is a special case of ${\bf{(A1^{\prime})}}$.
\end{remark}

 Base on Theorems \ref{dingli15}, we have the following.

\begin{theorem}
For Hamiltonian system (\ref{5110}), i.e., (\ref{5109}).
 \begin{itemize}
\item[{\bf (1)}] Under assumptions $\bf{(H3)} - \bf{(H6)}$ there exist a $\varepsilon_0 >0 $ and a family of Cantor sets $G_\varepsilon \subset G$, $0<\varepsilon < \varepsilon_0$, such that on $G_\varepsilon$ there is family of $n-$invariant tori for (\ref{5109}), i.e., there is a family of quasiperiodic solution with $n$ incommensurate frequencies for (\ref{5109}). Moreover, the relative Lebesgue measure $|G \setminus G_\varepsilon |$ tends to 0 as $\varepsilon \rightarrow 0$.
 \item[{\bf (2)}] Under assumptions $\bf{(H3)} - \bf{(H5)}$ and ${\bf{(H7)}}$ there exist a $\varepsilon_0 >0 $ and a family of Cantor sets $G_\varepsilon\subset \mathcal{M} \subset G$, $0<\varepsilon < \varepsilon_0$, such that on $G_\varepsilon$  for (\ref{5109}) there is a family of $n-$invariant tori whose frequencies are $\Lambda_\varepsilon$ preserved the ratio of the $i_1,$, $\cdots$, $i_n$ components of its toral frequency $\Lambda_\varepsilon(I)$ i.e., $$[\Lambda_{\varepsilon,i_1}(I), \cdots, \Lambda_{\varepsilon,i_n}(I)] = [\Lambda_{i_1}(I), \cdots, \Lambda_{i_n}(I)],$$ where $\Lambda_{\varepsilon,i_j}(I)$ and $\Lambda_{i_j}(I)$ are the $i_j$ components of $\Lambda_\varepsilon$ and $\Lambda$, respectively. Moreover, the relative Lebesgue measure $|G \setminus G_\varepsilon |$ tends to 0 as $\varepsilon \rightarrow 0$.
  \end{itemize}
\end{theorem}

\begin{proof}
For the sake of simplicity, we omit the proof and for details, refer to Sections \ref{couple} and \ref{5108}.

\end{proof}

\end{document}